\documentclass[11pt,a4paper]{article}
\usepackage[T1]{fontenc}
\usepackage[utf8]{inputenc}
\usepackage{authblk}
\usepackage{amssymb, amsmath}
\usepackage{amsthm, amsfonts,mathrsfs}
\usepackage[T1]{fontenc}
\usepackage{times}
\usepackage{graphicx}
\usepackage{subfigure}
\usepackage{cite}
\usepackage{hyperref}
\usepackage{epsfig}
\usepackage{color}
\usepackage[all]{xypic}
\hypersetup{colorlinks=true,citecolor=blue,linkcolor=blue}
\newtheorem{thm}{Theorem}[section]
\newtheorem{coro}[thm]{Corollary}

\newtheorem{lemma}[thm]{Lemma}
\newtheorem{prop}[thm]{Proposition} 
\newtheorem{remark}[thm]{Remark}

\numberwithin{equation}{section}
\newcommand{\ep}{\varepsilon}
\newcommand{\al}{\alpha}

\newcommand{\ue}{U^\varepsilon}
\newcommand{\bbue}{\mathbb{U}^\varepsilon}
\newcommand{\bfue}{\mathbf{U}^\varepsilon}
\newcommand{\ee}{\eta^\varepsilon}
\newcommand{\E}{\mathbb{E}}
\newcommand{\F}{\mathcal{F}}

\newcommand{\vp}{\varphi}

\newcommand{\xe}{X^\varepsilon}

\newcommand{\ze}{Z^\varepsilon}
\newcommand{\we}{W^\varepsilon}
\newcommand{\ve}{V^\varepsilon}
\newcommand{\me}{M^\varepsilon}

\newcommand{\intot}{\int_0^t}
\newcommand{\intos}{\int_0^s}
\newcommand{\supt}{\sup\limits_{0\leq t\leq T}}
\newcommand{\supe}{\sup\limits_{0<\varepsilon\leq 1}}

\newcommand{\bbxe}{\mathbb{X}^{\varepsilon}}
\newcommand{\bfxe}{\mathbf{X}^{\varepsilon}}
\newcommand{\epi}{\varepsilon^{-1}}
\newcommand{\bx}{\mathbb{X}}
\newcommand{\bfx}{\mathbf{X}}
\newcommand{\by}{\mathbb{Y}}
\newcommand{\bfy}{\mathbf{Y}}
\newcommand{\HH}{\mathcal{H}}
\newcommand{\LL}{\mathcal{L}}
\newcommand{\law}{\operatorname{Law}}
\title{\textbf{Approximation of Random Differential Equations Driven by a Singularly Perturbed Fast Oscillating Process}\thanks{This work is supported by NSFC Grant No.  12371243.}}

\author[1]{Qingming Zhao\thanks{qingming.zhao@smail.nju.edu.cn}}

\author[2]{Xueru Liu\thanks{327625941@qq.com}}

\author[1]{Wei Wang\thanks{Corresponding author: wangweinju@nju.edu.cn}}

\affil[1]{School of Mathematics, Nanjing University, Nanjing 210093, P. R. China}
\affil[2]{School of Mathematical Sciences, Dalian University of Technology, 116024 Dalian, P. R. China}

\date{} 

\begin{document}
	\maketitle

	\noindent{\small{\hspace{1.1cm} }}

	\noindent \textbf{Abstract~~~}{We investigate the approximation of random differential equations driven by a singularly perturbed fast oscillating process, whose dynamics is governed by a second-order equation with scaled mixing random force. By a diffusion approximation approach, we explore the limit of the rough path lift of this semimartingale, and a universal limit theorem is applied to identify the limit of random differential equation. A structurally parallel proof also applies to establish an iterated weak invariance principle for the mixing random force, which is itself an independent interesting result. We find that, the limits of both of the second-level processes, have the form of iterated integral of Stratonovich form plus an anti-symmetric part which is proportional to the time increment.   }   
	\\[2mm]
	\textbf{Keywords~~~} {Diffusion approximation}, {Iterated weak invariance principle}, {Smoluchowski--Kramers approximation}, {Singular perturbation}, {Rough path theory}
	\\[2mm]
	\\
	\textbf{2020 Mathematics Subject Classification~~~}60L90, 60H10

	\section{Introduction}
	On a complete probability space $(\Omega,\F,P),$ we consider the approximation of the following stochastic differential equations when $\ep\to 0$
	\begin{equation}\label{prelimit}
		d\ze_t=b(\ze_t)dt+\sigma(\ze_t)d\xe_t, \quad\ze_0=z,
	\end{equation}
	where $z\in\mathbb{R}^e$ is deterministic, $\xe$ itself is a physical process, satisfying
	\begin{equation}\label{SKx}
		\ep\ddot{X}^\ep+\dot{X}^\ep=F(\xe)+\dot{U}^\ep,\quad \xe_0=0,\dot{X}^\ep_0=0,
	\end{equation}
	or equivalently,
	\begin{equation}\label{SKxv}
		\begin{cases}
			\dot{X}^\ep=\ve,\enspace\xe_0=0,\\
			\dot{V}^\ep=-\frac{1}{\ep}\ve+\frac{1}{\ep}F(\xe)+\frac{1}{\ep}\dot{U}^\ep,\enspace\ve_0=0.
		\end{cases}
	\end{equation}
	Here, $\ue_t=\frac{1}{\sqrt{\ep}}\intot\ee(s)ds$, $\ee(s):=\eta(\frac{s}{\ep})$, and $\eta$ is a strictly stationary process, having trajectories in $C^1(\mathbb{R};\mathbb{R}^d)$ a.e.. The detailed assumptions of $\eta$ are presented in Section \ref{section-preli}. As already discussed in \cite{LW23}, $\xe\to X$ in $C([0,T];\mathbb{R}^d)$ weakly for some continuous random process $X$, and heuristically, there should be $\ze\to Z$ with 
	$$dZ_t=b(Z_t)dt+\sigma(Z_t)dX_t,\quad Z_0=z.$$  However, this is in general incorrect, as one sees in the classical Wong--Zakai approximation: the sequence of It\^o SDE driven by piecewise linear approximation $B^n$ of a Brownian motion $B$ converges to the Stratonovich SDE driven by $B$. This is explained, in the rough path framework, by the fact that the It\^o lifts of $B^n$ converge to the Stratonovich lift of $B$. It is possible to investigate approximation of \eqref{prelimit} only using traditional stochastic analysis \cite[e.g.]{IW89}, but a rather heavy computation is involved. The universal limit theorem (also known as continuity of It\^o--Lyons map) in rough path framework indicates that, in order to investigate the limit of \eqref{prelimit}, it is sufficient to consider the lift of $\xe$ in the rough path metric. Recently there is fruitful work on approximating SDE in rough path framework, see for instance in  \cite{KM16,CFK+16, EFO24, FK24}.
	
	 Equations \eqref{SKxv} imply that the trajectory of $\xe$ is of bounded variation. Therefore, the integrals $$\bbue_{s,t}:=\int_s^t\ue_{s,r}\otimes d\ue_r,\enspace \mathbb{X}^\ep_{s,t}:=\int_s^t\xe_{s,r}\otimes d\xe_r,\quad 0\leq s\leq t<\infty,$$ are both well-defined in the sense of Riemann--Stieltjes. Set $\bfue:=(\ue,\bbue),\quad\bfxe:=(\xe,\bbxe).$ $\bfue$ and $\bfxe$ are known as canonical lifts of $\ue$ and $\xe$, respectively. 
	 
	 We explore the limit $\mathbf{U}$ of $\bfue$ first. The approximation of the first-level process $\ue$ has been studied in \cite{DW14,SKO89}. But to our best knowledge, there is no work considering the limit of the lifted process $\bfue$. We show that, under some mixing condition on $\eta$, $\bfue$ converges weakly to a Brownian rough paths with some parameters. This is called an iterated weak invariance principle (iterated WIP), which draw more and more attention in recent years \cite[e.g.]{KM16}. Logically, one cannot derive the limit $\mathbf{X}$ of $\bfxe$ by a direct use of universal limit theorem together with the convergence result of $\bfue$, since $(\xe,\ve)$ is a system with degenerate noise $(0,\frac{1}{\ep}\ue)$, the second component of which blows up when $\ep\to0$, as seen in \eqref{SKxv}. The derivation of $\mathbf{X}$ shares some structural parallelism to that of $\mathbf{U}$, but the former one is more subtle.
	 
	   Next we consider the limit of $\bfxe$. There has been numerous work on the approximation of the first-level process $\xe$. In \cite{WR15, SW21, LW23}, approximation of singularly perturbed state-dependent Langevin equation, generalized Langevin equation and stochastic Burgers equations are considered respectively, all of which are driven by fast oscillating random force. This regime of limit is also known as Smoluchowski--Kramers (SK) approximation. For SK approximation with Brownian noise, we refer to \cite{HMV15, CX22, SWW24}. However, for the second-level process, there seems to be few references. In \cite{FGL15}, they considered the linear equations describing physical Brownian motion in magnetic fields . In our paper, we do not assume the driving random force in $\xe$ is Gaussian, but only make some mixing condition hypothesis. We show that, as $\ep\to 0$, the second-level process $\bbxe$ converges to  the Stratonovich iterated integral of the limit equation, plus an additional purely deterministic anti-symmetric part involving the covariance of $\eta$. To this aim, tightness of $(\bfxe)_\ep$ is needed, which follows from Kolmogorov criterion for rough paths. This requires us to give $p$-th moment estimate for $\xe$ and $\bbxe$ for sufficiently large $p$. It seems that there is no existing literature showing this, and we prove it by a Gordin martingale coboundary decomposition approach. 
	 
	 Our article is organized as follows. In Section \ref{section-preli}, we recall basic rough path theory, impose several assumptions and state main theorems. In Section \ref{section-moment}, we give moment estimate, from which the desired tightness is obtained. Section \ref{section-limitofU} devotes to establishing Theorem \ref{main1} by an approach of diffusion approximation. By a more subtle analysis, we prove Theorem \ref{main2} in Section \ref{section-limitofX}, and an application of universal limit theorem yields Theorem \ref{main3}.
	
	\section{Preliminaries and Main Results}\label{section-preli}
	We first recall basic notions in rough path theory \cite[e.g.]{FH20}. Let $\mathcal{C}^\al([0,T];\mathbb{R}^d)$ be the space of $\mathbb{R}^d$-valued functions on $[0,T]$ such that $\|X\|_{\al-\operatorname{H\"ol}}:=\sup\limits_{0\leq s< t\leq T}\frac{|X_{s,t}|}{|t-s|^\al}<\infty$, where $X_{s,t}:=X_t-X_s$. Setting  $$\Delta_T:=\{(s,t)\in \mathbb{R}^2:0\leq s\leq t\leq T\},$$ the space $\mathcal{C}^{2\al}_2(\Delta_T;\mathbb{R}^{d\times d})$ is defined by a set consisting of functions $A:\Delta_T\to\mathbb{R}^{d\times d}$ such that $$\|A\|_{2\al-\text{H\"ol}}:=\sup_{0\leq s<t\leq T}\frac{|A_{s,t}|}{|t-s|^{2\al}}<\infty.$$ The $\al$-H\"older rough path space over $\mathbb{R}^d$, denoted by $\mathscr{C}^\al([0,T];\mathbb{R}^d)$, is the set consisting of elements $\mathbf{X}=(X,\mathbb{X})$ in $\mathcal{C}^\al([0,T];\mathbb{R}^d)\times \mathcal{C}^{2\al}_2(\Delta_T;\mathbb{R}^{d\times d})$ with $X_0=0,$ such that 
	\begin{equation}\label{chen}
		\mathbb{X}_{s,t}=\mathbb{X}_{s,u}+\mathbb{X}_{u,t}+X_{s,u}\otimes X_{u,t},\quad 0\leq s\leq u\leq t\leq T,
	\end{equation}
	 endowed with the rough path metric
	 \begin{equation*}
	 	\rho_\al(\mathbf{X},\mathbf{Y}):=\sup_{0\leq s< t\leq T}\frac{|X_{s,t}-Y_{s,t}|}{|t-s|^\al}+\sup_{0\leq s<t\leq T}\frac{|\mathbb{X}_{s,t}-\mathbb{Y}_{s,t}|}{|t-s|^{2\al}},\quad \mathbf{X},\mathbf{Y}\in\mathscr{C}^{\al}([0,T];\mathbb{R}^d).
	 \end{equation*}
	 The non-linear equation \eqref{chen} is known as Chen relation. Throughout the paper, we write 
	 \[\bx_t:=\bx_{0,t},\enspace\bfx_t:=(X_t,\bx_t),\quad \bfx=(X,\bx)\in\mathscr{C}^\al([0,T];\mathbb{R}^d).\]
	 
	  A classical fact \cite[Section 5.3]{FV10} is that $\mathcal{C}^\al([0,T];\mathbb{R}^d)$ is not separable, but $\mathcal{C}^{0,\al}([0,T];\mathbb{R}^d)$, the closure of smooth path space in the $\al$-H\"older space, is Polish. Similarly, a drawback of rough path space is the absence of separability, making it inconvenient to use tightness property. To fix this, introduce the space of smooth rough paths $\mathscr{C}^\infty([0,T];\mathbb{R}^d)$\cite[Page 17]{FH20}: $\mathbf{X}=(X,\mathbb{X})\in\mathscr{C}^\infty$ if and only if $\mathbf{X}\in\mathscr{C}^\al$, $X\in C^\infty$ and $\mathbb{X}_{s,\cdot}\in C^\infty$ for each base-point $s\in [0,T].$ Consider the space $\mathscr{C}^{0,\al}([0,T];\mathbb{R}^d)$, defined by the closure of $\mathscr{C}^\infty([0,T];\mathbb{R}^d)$ in $\mathscr{C}^\al([0,T];\mathbb{R}^d)$ with respect to the rough path metric $\rho_\al$. It is known that \cite[Exercise 2.8]{FH20} $\mathscr{C}^{0,\al}([0,T];\mathbb{R}^d)$ is a Polish space.
	 
	 We also recall the definition of rough differential equations (RDE) \cite[Section 8.7]{FH20}.    Let $b\in C^3_b(\mathbb{R}^e;\mathbb{R}^e)$, $\sigma\in C^3_b(\mathbb{R}^e;\mathbb{R}^{e\times d})$, the space of functions having bounded continuous derivatives of $k$-orders, $k=0,1,2,3$.  We say $Y$ with $Y_0=y$ is a solution of 
	 \begin{equation*}
	 	dY_t=b(Y_t)dt+\sigma(Y_t)d\mathbf{X}_t,\quad Y_0=y,
	 \end{equation*}
	 if
	 $$Y_t-Y_s=b(Y_s)(t-s)+\sigma(Y_s)X_{s,t}+D\sigma(Y_s)\sigma(Y_s)\mathbb{X}_{s,t}+R_{s,t}, \quad (s,t)\in\Delta_{T},$$
	 and $$\lim\limits_{\delta\to 0}\sup\limits_{\mathcal{P}\subset [0,T],|\mathcal{P}|<\delta}\sum\limits_{[s,t]\in\mathcal{P}}|R_{s,t}|=0,$$ where $\mathcal{P}\subset[0,T]$ means that $\mathcal{P}$ is a partition of $[0,T]$ with mesh $|\mathcal{P}|$. One of the most important properties of RDE is that the It\^o--Lyons map is locally Lipschitz with respect to $\rho_\al$ \cite[Theorem 8.5]{FH20}.	We refer to \cite{LCL07,LQ02,FH20,FV10} for systematic study of rough path theory.

	 Now we turn to the setup of our questions. Denote
	 \[\F_s^t:=\sigma(\eta(r):s<r\leq t),\quad -\infty\leq s\leq t\leq \infty, \] and set
	 \[m(t):=\sup\limits_{s\in\mathbb{R}}\alpha(\mathcal{F}_{-\infty}^s,\mathcal{F}_{s+t}^\infty),  \]
	   here $\alpha$ is the strong mixing coefficient, that is, $$\alpha(\mathcal{A},\mathcal{B}):=\sup\limits_{A\in\mathcal{A},B\in\mathcal{B}}|P(AB)-P(A)P(B)|.$$
	   For each $t\in\mathbb{R},$ we also set
	   \[\HH_t:=\sigma(\eta(r):-\infty<r\leq t)\quad\text{with usual augmentation.}\] 
	   It is known that \cite[(67.4) in Chapter II]{RW00} $\HH_t$ can be obtained by the following procedure:
	   \[\hat{\F}_{-\infty}^t:=\F_{-\infty}^t\vee\mathcal{N}:=\sigma(\F_{-\infty}^t\cup\mathcal{N}),\quad \HH_t:=\cap_{h>0}\hat{\F}_{-\infty}^{t+h},\] here $\mathcal{N}$ collects all the null-set. But by continuity and stationarity of $\eta,$ $(\hat\F_{-\infty}^t)_t$ is already right-continuous \cite[Remark 2.3.5]{ARN98}, so
	   \begin{equation}\label{augment-v.s.-usual}
	   	\HH_t=\hat{\F}_{-\infty}^{t},\quad t\in\mathbb{R}.
	   \end{equation}
	 
	 We impose following conditions:
	
	\noindent($\mathbf{A}_1$) $\E\eta(t)=0$ and $\E|\eta(t)|^p\leq C_p$ for each $t\in\mathbb{R}$ and $p\geq 1.$
	
	\noindent($\mathbf{A}_2$) There is a constant $C$ such that for all $t\geq0$, $m(t)\leq Ce^{-Ct}$.
	
	\noindent($\mathbf{A}_3$) $F:\mathbb{R}^d\to\mathbb{R}^d$ is globally Lipschitz and bounded.

Set $R(t):=\E\eta(0)\otimes\eta(t)$, and  $$A^{ij}:=\int_0^\infty R^{ij}(s)ds=\int_0^\infty \mathbb{E}[\eta^i(0)\eta^j(s)]ds,\enspace 1\leq i,j\leq d.$$ The integral is well-defined. Indeed, by Davydov inequality (see Proposition \ref{prop-davydov} below) and ($\mathbf{A}_1$),
$$|\mathbb{E}[\eta^i(0)\eta^j(s)]|=|\text{Cov}(\eta^i(0),\eta^j(s))|\lesssim\al(\F_{-\infty}^0,\F_s^\infty)^{1/4}\|\eta^i(0)\|_2\|\eta^j(s)\|_4\lesssim m(s)^{1/4},$$ so by ($\mathbf{A}_2$),
$$|A^{ij}|\lesssim\int_0^\infty m(s)^{1/4}ds<\infty.$$	

Now we state our main results. The first theorem establishes an iterated WIP. 
\begin{thm}\label{main1}
	Let $\al\in(\frac{1}{3},\frac{1}{2}),$ and $S\in\mathbb{R}^{d\times d}$ be the matrix with $(i,j)$-entry $S^{ij}:=A^{ij}+A^{ji}.$
	Under assumptions ($\mathbf{A}_1$)--($\mathbf{A}_2$), $\bfue\to\mathbf{U}$ in $\mathscr{C}^{\al}([0,T];\mathbb{R}^d)$ weakly as $\ep\to 0$, where $\mathbf{U}=(U,\mathbb{U})$,  $$U_t=S^{1/2}B_t,\quad \mathbb{U}_{s,t}=\int_s^t U_{s,r}\otimes\circ dU_r+\frac{A-A^\top}{2}(t-s).$$
\end{thm}

The next theorem gives a Smoluchowski--Kramers approximation of the rough path lift.
\begin{thm}\label{main2}
  Let $\al\in(\frac{1}{3},\frac{1}{2}).$ Set
  \[\Gamma^{ij}:=\int_0^\infty\int_0^\infty e^{-u}e^{-v}R^{ij}(u-v)dudv,\quad \Delta^{ij}:=\int_0^\infty\int_0^\infty e^{-u}R^{ij}(u+v)dudv\]
    and \[\bar{\Gamma}^{ij}:=\Gamma^{ij}+\Delta^{ij}.\] Let $\bar{\Gamma}\in\mathbb{R}^{d\times d}$ be the matrix with $(i,j)$-entry $\bar{\Gamma}^{ij}$. Under assumptions ($\mathbf{A}_1$)--($\mathbf{A}_3$), $\mathbf{X}^\ep\to \mathbf{X}$ in $\mathscr{C}^\al([0,T];\mathbb{R}^d)$ weakly, where $\mathbf{X}=(X,\mathbb{X})$, $X_0=0$ and
  \begin{equation}\label{limitofX}
  	dX_t=F(X_t)dt+dU_t,\quad\mathbb{X}_{s,t}=\int_s^t X_{s,r}\otimes\circ dX_r+\frac{\bar{\Gamma}-\bar{\Gamma}^\top}{2}(t-s).
  \end{equation}
\end{thm}

An application of universal limit theorem yields the following result.
\begin{thm}\label{main3}
	Let $\al\in(\frac{1}{3},\frac{1}{2}),$ $b\in C^3_b(\mathbb{R}^e;\mathbb{R}^e),$ and $\sigma\in C^3_b(\mathbb{R}^e;\mathbb{R}^{e\times d}).$ Under assumptions ($\mathbf{A}_1$)--($\mathbf{A}_3$), as $\ep\to0$, the solution of \eqref{prelimit} $Z^\ep\to Z$ weakly in $\mathcal{C}^\al([0,T];\mathbb{R}^e)$ with $Z_0=z$ and 
	$$dZ^i_t=[b^i(Z_t)+\sum_{k,l,j}\partial_k\sigma_{ij}(Z_t)\sigma_{kl}(Z_t)G^{lj}+\sum_j\sigma_{ij}(Z_t)F^j(X_t)]dt+\sum_{j,k} S^{1/2}_{jk}\sigma_{ij}(Z_t)dB^k_t,$$
	where $G^{lj}:=\frac{1}{2}(\sum\limits_k S^{1/2}_{lk}S^{1/2}_{jk}+\bar{\Gamma}_{lj}-\bar{\Gamma}_{jl})$, and $S^{1/2}_{mn}$ is the $(m,n)$-component of $S^{1/2}$.
\end{thm}

	\section{Moment Estimates}\label{section-moment}
	From now on, we write $\|\zeta\|_p:=(\E|\zeta|^p)^{1/p}$ for random vector $\zeta$. 
	\begin{prop}\label{prop-momentestimateforv}
		Let $\we:=\sqrt{\ep}\ve.$ For each $0\leq t\leq T$ and $p\geq 1$, there exists a constant $C_{p,T}$ such that
		\begin{equation*}
			\supe\supt\|\we_t\|_p\leq C_{p,T}.
		\end{equation*}
		Equivalently,
		\begin{equation*}
			\supt\|\ve_t\|_p\leq C_{p,T}\ep^{-1/2},\quad 0<\ep\leq 1.
		\end{equation*}
	\end{prop}
	\begin{proof}
		From (\ref{SKxv}) one solves that
		\begin{equation*}
			\ve_s=\epi\intos e^{-\epi(s-r)}d\ue_r+\epi\intos e^{-\epi(s-r)}F(\xe_r)dr.
		\end{equation*}
		Multiplying by $\sqrt{\ep},$
		\begin{equation}\label{decompowe}
			\we_s=\frac{1}{\ep}\int_0^s e^{-\epi(s-r)}\eta(r/\ep)dr+\frac{1}{\sqrt{\ep}}\intos e^{-\epi(s-r)}F(\xe_r)dr=:z^\ep_s+R^\ep_s.
		\end{equation}	
		By change-of-variable,$$z^\ep_s=\int_0^{s/\ep}e^{-u}\eta(\frac{s}{\ep}-u)du,\enspace R^\ep_s=\ep^{1/2}\int_0^{s/\ep}e^{-u}F(\xe_{s-\ep u})du,$$ so
		\begin{equation}\label{esforze}
			\|z^\ep_s\|_p\leq\int_0^{s/\ep}\|e^{-u}\eta(\frac{s}{\ep}-u)\|_pdu\leq \|\eta(0)\|_p\int_0^\infty e^{-u}du\leq C_p,
		\end{equation} 
		and the boundedness of $F$ implies
		\begin{equation}\label{esforre}
			\|R^\ep_s\|_\infty\lesssim \ep^{1/2}\int_0^{\infty}e^{-u}du=\mathcal{O}(\ep^{1/2}).
		\end{equation}
		The proof is finished by combining \eqref{esforze} and \eqref{esforre}.
	\end{proof}
	We need a moment inequality for sum of mixing sequence, which is a slightly less general version of \cite[Theorem 3]{YOS78}. 
	\begin{lemma}\label{lemma-mixingsum}
		Given an even integer $r\geq 2$ and $\delta>0$. Let $(X_j)_{j\in\mathbb{Z}}$ be a real centered strictly stationary $\al$-mixing sequence. Assume that
		$\E|X_1|^{r+\delta}<\infty$, and 
		\begin{equation}\label{yokoyamacondition}
			\sum_{i=0}^\infty(i+1)^{\frac{r}{2}-1}[\al(i)]^{\frac{\delta}{r+\delta}}<\infty,
		\end{equation}
		where $$\al(n):=\sup_{k\in\mathbb{N}}\al(\sigma(X_j:j\leq k),\sigma(X_j:j\geq k+n)).$$
		 Then there is a constant $C$ such that for all $n\in\mathbb{N}$,
		$$\E\Big(\sum_{j=1}^{n}X_j\Big)^r\leq Cn^{r/2}.$$
	\end{lemma}
	\begin{remark}[Non-centered case]
		If $X_j$ is not centered, set $\tilde{X_j}=X_j-\E X_j,$ then
		\begin{equation}\label{Yokoyama2}
			\E\Big|\sum_{j=1}^n X_j\Big|^r=\E\Big|\sum_{j=1}^{n} \tilde{X_j}+\sum_{j=1}^nEX_j\Big|^r\leq C_r n^{r/2}+C_r n^r\leq C_r n^{r}.
		\end{equation}
	\end{remark}
	
	We need the following Gordin decomposition theorem \cite{GOR69} (see also \cite[Theorem 1]{VOL06}, which is slightly stronger than what we stated below).
	
	\begin{lemma}[Gordin decomposition]\label{lemma-Gordin}
		Let $(Y_n)_{n\in\mathbb{Z}}$ be a strictly stationary sequence, $\mathcal{G}_n:=\sigma(Y_k:k\leq n)$. Assume that $1\leq r<\infty$ and for each $n\in\mathbb{Z},$
		\begin{equation}\label{conditionforGordin}
			\E |Y_0|^r<\infty, \quad\sum_{k=0}^{\infty} \|\E(Y_{n+k}|\mathcal{G}_{n-1})\|_r<\infty. 
		\end{equation}
		Then there exist random sequences $(d_n)_{n\in\mathbb{Z}}$ and $(Z_n)_{n\in \mathbb{Z}}$ in $L^r$, such that
		
		(i) $Z_n$ is $\mathcal{G}_{n-1}$ measurable, $n\in\mathbb{Z}.$
		
		(ii) $d_n$ is $\mathcal{G}_n$ measurable, and $(d_n,\mathcal{G}_n)$ is a martingale difference sequence, that is,
		\begin{equation*}
			\E(d_{n+1}|\mathcal{G}_n)=0,\quad n\in\mathbb{Z}.
		\end{equation*}.
		
		(iii)\begin{equation*}
			\sup_{n\in\mathbb{Z}}(\|Z_n\|_r+\|d_n\|_r)<\infty
		\end{equation*}
		
		(iv)\begin{equation*}
			Y_n=d_{n}+Z_n-Z_{n+1},\quad n\in\mathbb{Z}.
		\end{equation*}
	\end{lemma}
	We also need the following Davydov inequality \cite[Coro 16.2.4]{AL06}.
	\begin{prop}[Davydov inequality]\label{prop-davydov}
		Let $X, Y$ be two real random variables. Let $\al:=\al(\sigma(X),\sigma(Y))$. Suppose that $\E|X|^p\vee\E|Y|^q<\infty$, where $p,q>1$ and $\frac{1}{p}+\frac{1}{q}<1$. Set $\frac{1}{r}=1-\frac{1}{p}-\frac{1}{q}$, then
		$$|\operatorname{Cov}(X,Y)|\leq 2r(2\al)^{1/r}\|X\|_p\|Y\|_q.$$
	\end{prop}	
	The following lemma is obtained by a direct application of Davydov inequality.
	\begin{lemma}
		Let $\mathcal{A},$ $\mathcal{B}$ be two sigma-fields, and $Y$ is a $\mathcal{B}$-measurable real random variable with \[\E Y=0,\quad \|Y\|_k<\infty\] for some $k>p>1.$ Then
		\begin{equation}\label{general-case}
			\|\E(Y\mid\mathcal{A})\|_p\leq C_{p,k}\al(\mathcal{A},\mathcal{B})^{\frac{1}{p}-\frac{1}{k}}\|Y\|_k.
		\end{equation}
		In particular, if  $Y$ is a $\F_{s+t}^\infty$-measurable real random variable, $\E Y=0,$ $\|Y\|_k<\infty,$ $k>p> 1$, then
		\begin{equation}\label{continuous-time-case}
			\|\E(Y|\F_{-\infty}^s)\|_p\leq C_{p,k} m(t)^{\frac{1}{p}-\frac{1}{k}}\|Y\|_k.
		\end{equation}
	\end{lemma}

	\begin{proof}
		Set $q=\frac{p}{p-1},$ so that $\frac{1}{q}+\frac{1}{k}<\frac{1}{q}+\frac{1}{p}=1.$ By dual representation of $L^p,$
		\begin{eqnarray}\label{eqforlemma1}
			\|\E(Y|\mathcal{A})\|_p&=&\sup_{X\in L^q(\mathcal{A}),\|X\|_q\leq 1}|\E[X\E(Y|\mathcal{A})]|\nonumber\\&=&\sup_{X\in L^q(\mathcal{A}),\|X\|_q\leq 1}|\E\E(XY|\mathcal{A})|\nonumber\\&=&\sup_{X\in L^q(\mathcal{A}),\|X\|_q\leq 1}|E(XY)|\nonumber\\&=&\sup_{X\in L^q(\mathcal{A}),\|X\|_q\leq 1}|\operatorname{Cov}(X,Y)|.
		\end{eqnarray}
		Let $\frac{1}{r}=1-\frac{1}{q}-\frac{1}{k}.$ By Davydov inequality, for each $\mathcal{A}$-measurable random variable $X$ with $\|X\|_q\leq 1$,
		\begin{eqnarray}\label{eqforlemma2}
			|\operatorname{Cov}(X,Y)|&\lesssim_r&[\al(\sigma(X),\sigma(Y))]^{1/r}\|X\|_q\|Y\|_k\nonumber\\&\leq&[\al(\mathcal{A},\mathcal{B})]^{1/r}\|X\|_q\|Y\|_k\nonumber\\&\leq& \al(\mathcal{A},\mathcal{B})^{1/r}\|Y\|_k.
		\end{eqnarray}
		Combining (\ref{eqforlemma1})--(\ref{eqforlemma2}), the proof of \eqref{general-case} is finished. In particular, if  $Y$ is $\F_{s+t}^\infty$-measurable with $\E Y=0,$ and $\|Y\|_k<\infty,$ for some $k>p>1,$ we have by \eqref{general-case}
		\[\|\E(Y|\F_{-\infty}^s)\|_p\leq C_{p,k} \al(\F_{-\infty}^s,\F_{s+t}^\infty)^{\frac{1}{p}-\frac{1}{k}}\|Y\|_k\leq C_{p,k}m(t)^{\frac{1}{p}-\frac{1}{k}}\|Y\|_k,\] which is precisely \eqref{continuous-time-case}.

	\end{proof}
	The main proposition of this section is as follows. 
	\begin{prop}\label{prop-momentestimate}
		For each $0<\ep\leq 1$ ,$0\leq s\leq t\leq T,$ and $p\geq 1$, there exists a constant $C$ independent of $\ep, s$ and $t$, such that
		
		(i) $$\|\ue_{s,t}\|_p\leq C|t-s|^{1/2},\quad \|\bbue_{s,t}\|_p\leq C|t-s|.$$
		
		(ii) $$\|\xe_{s,t}\|_p\leq C|t-s|^{1/2},\quad \|\bbxe_{s,t}\|_p\leq C|t-s|.$$
	\end{prop}
	Before proving Proposition \ref{prop-momentestimate}, we present an abstract result.
	\begin{prop}\label{prop-momentestimateformixing}
		Let $\xi$ be a centered, strictly stationary process with almost surely continuous trajectory, and $$\E|\xi(0)|^p<\infty,\quad\text{for all}\enspace p\geq 1.$$  Set
		\[\F_{\xi,a}^b:=\sigma(\xi(r):a<r\leq b),\quad m_\xi(u):=\sup_{s\in\mathbb{R}}\al(\F_{\xi,-\infty}^s,\F_{\xi,s+u}^\infty),\]
		and assume
		\begin{equation}\label{assumption-xi}
			m_\xi(u)\leq Ce^{-cu},\quad u\geq 0.
		\end{equation}
		 Let $a,b\in[0,\infty)$ with $a\leq b$, then there exists a constant $C_p$  such that

		(i)\begin{equation*}
			\Big\|\int_a^b\xi(u)du\Big\|_p\leq C_p (b-a)^{1/2}.
		\end{equation*}
		
		(ii)\begin{equation*}
			\Big\|\int_a^b\int_a^v\xi(u)\otimes \xi(v)dudv\Big\|_p\leq C_p(b-a).
		\end{equation*}
	\end{prop}

	\begin{proof}
		By stationarity, it suffices to show
		\begin{equation}\label{target}
			\Big\|\int_0^N\xi(u)du\Big\|_p\leq C_p N^{1/2},\quad N\geq 0,
		\end{equation}
		and
		\begin{equation}\label{target2}
			\Big\|\int_0^N\int_0^v\xi(u)\otimes \xi(v)dudv\Big\|_p\leq C_pN, \quad N\geq 0.
		\end{equation}
		Note that \eqref{target}--\eqref{target2} are trivial when $0\leq N<1$. Indeed,
		\begin{equation}\label{small1}
			\Big\|\int_0^N\xi(u)du\Big\|_p\leq\int_0^N\|\xi(u)\|_pdu=C_pN\leq C_p N^{1/2},\quad 0\leq N<1.
		\end{equation}
		and
		\begin{eqnarray}\label{small2}
			&&\Big\|\int_0^N\int_0^v\xi(u)\otimes\xi(v)dudv\Big\|_p\nonumber\\&\leq&\int_0^N\int_0^v\|\xi(u)\otimes\xi(v)\|_pdudv\nonumber\\&\leq&\int_0^N\int_0^v\|\xi(u)\|_{2p}\cdot\|\xi(v)\|_{2p}dudv\nonumber\\&\leq&C_p\int_0^N\int_0^v C_pdudv\nonumber\\&=&C_p\frac{N^2}{2}\leq C_pN,\quad 0\leq N<1.
		\end{eqnarray}
		In the following, we always assume $N\geq1$.
		
		\noindent\textbf{Proof of} (i)\textbf{:} 
		Without loss of generality we assume $p$ is an even integer. If $N\in\mathbb{N}$, by Lemma \ref{lemma-mixingsum}, 
		\begin{equation*}
			\E\Big|\int_0^N\xi(u)du\Big|^p=\E\Big|\sum_{k=0}^{N-1}\int_k^{k+1}\xi(u)du\Big|^p\leq C_pN^{p/2}.
		\end{equation*}
		For general $N\geq 1$, write
		\begin{equation}\label{newint0n}
			\int_0^N\xi(u)du=\int_0^{\lfloor N\rfloor}\xi(u)du+\int_{\lfloor N\rfloor}^N\xi(u)du.
		\end{equation}
		By stationarity and \eqref{small1}, 
		\begin{equation}\label{smallpart}
		\E\Big|\int_{\lfloor N\rfloor}^N\xi(u)du\Big|^p\leq C_p N^{p/2}.
		\end{equation}
		From the integer case which is already proved,
		\begin{equation}\label{integerpart}
			\E\Big|\int_0^{\lfloor N\rfloor}\xi(u)du\Big|^p\leq C_pN^{p/2}.
		\end{equation}
		Combining \eqref{newint0n}--\eqref{integerpart},
		\eqref{target} is proved.
		
		\noindent\textbf{Proof of} (ii)\textbf{:}

		 We first assume $N\in\mathbb{N}$.  Partitioning the integral region, we write
		\begin{eqnarray}\label{noname01}
			\int_0^N\int_0^v\xi(u)\otimes \xi(v)dudv=I_1+I_2,
		\end{eqnarray}
		where 
		\begin{eqnarray*}
			I_1&:=&\sum\limits_{m=0}^{N-1}\int_m^{m+1}\int_m^v\xi(u)\otimes\xi(v)dudv,\\I_2&:=&\sum\limits_{0\leq m<n\leq N-1}\int_m^{m+1}\xi(u)du\otimes\int_n^{n+1}\xi(v)dv.
		\end{eqnarray*}
		Applying \eqref{Yokoyama2} componentwise,
		\begin{eqnarray}\label{noname02}
			\E|I_1|^p\leq C_p N^p.
		\end{eqnarray}
		If $N=1,$ \eqref{noname01} gives $I_2=0.$ So we assume $N\geq 2.$ Set $Y_m:=\int_m^{m+1}\xi(u)du, m\in\mathbb{Z}$. Set \[\mathcal{G}_n:=\sigma(Y_j:j\leq n),\,\,\tilde{\mathcal{G}}_n:=\sigma(Y_j:j\geq n),\quad n\in\mathbb{Z}.\]
		By \eqref{general-case} with $p$ replaced by $2p$ and $k=4p$,
		\[ \|\E(Y_{n+l}\mid\mathcal{G}_{n-1})\|_{2p}\leq C_{p}\al(\mathcal{G}_{n-1},\tilde{\mathcal{G}}_{n+l})^{\frac{1}{4p}}\|Y_{n+l}\|_{4p}=C_p\al(\mathcal{G}_{n-1},\tilde{\mathcal{G}}_{n+l})^{\frac{1}{4p}} ,  \]
		but $$\mathcal{G}_{n-1}\subset\mathcal{F}_{\xi,-\infty}^{n},\quad \tilde{\mathcal{G}}_{n+l}\subset\mathcal{F}_{\xi,n+l}^\infty, $$
		so we have
		\[\|\E(Y_{n+l}\mid\mathcal{G}_{n-1})\|_{2p}\leq C_p\al(\F_{\xi,-\infty}^{n},\F_{\xi,n+l}^\infty)^{\frac{1}{4p}}\leq C_p m_\xi(l)^{\frac{1}{4p}}, \quad l\geq 0.\]
		Therefore, by \eqref{assumption-xi},
		\[\sum_{l=0}^\infty\|\E(Y_{n+l}\mid\mathcal{G}_{n-1})\|_{2p}\leq  C_p \sum_{l=0}^\infty m_\xi(l)^{\frac{1}{4p}}<\infty.\]
		Consequently, by Gordin decomposition, $Y_m=d_m+Z_m-Z_{m+1}$, where $(d_m)$ and $(Z_m)$ are sequences uniformly bounded in $L^{2p}$, and $(d_m)$ is a martingale difference sequence. We write
		$$\sum_{j=0}^{n-1}Y_j=\sum_{m=0}^{n-1}[d_m+Z_m-Z_{m+1}]=D_{n-1}+Z_0-Z_n,$$ where $D_n:=\sum_{j=0}^n d_j$. Therefore,
		\begin{eqnarray}\label{newI_2}
			I_2&=&\sum_{n=1}^{N-1}(\sum_{j=0}^{n-1}Y_j)\otimes Y_n\nonumber\\&=&\sum_{n=1}^{N-1}(D_{n-1}+Z_0-Z_n)\otimes (d_n+Z_n-Z_{n+1})\nonumber\\&=&\sum_{n=1}^{N-1}D_{n-1}\otimes d_n+\sum_{n=1}^{N-1} D_{n-1}\otimes(Z_n-Z_{n+1})\nonumber\\&&+\sum_{n=1}^{N-1}(Z_0-Z_n)\otimes d_n+\sum_{n=1}^{N-1}(Z_0-Z_n)\otimes(Z_n-Z_{n+1})\nonumber\\&=:&A+B+C+D.
					\end{eqnarray}
		By Minkowski and H\"older inequality,
		\begin{eqnarray}\label{D}
			\|D\|_p&\leq&\sum_{n=1}^{N-1}\|(Z_0-Z_n)\otimes(Z_n-Z_{n+1})\|_p\nonumber\\&\leq&\sum_{n=1}^{N-1}\|Z_0-Z_n\|_{2p}\cdot\|Z_n-Z_{n+1}\|_{2p}\nonumber\\&\leq&\sum_{n=1}^{N-1}(\|Z_0\|_{2p}+\|Z_n\|_{2p})\cdot(\|Z_n\|_{2p}+\|Z_{n+1}\|_{2p})\nonumber\\&\leq&C_p(N-1)\leq C_pN,
		\end{eqnarray}
		where we used uniform $L^{2p}$-boundedness of $(Z_n)$ in the second-last inequality.
		
		A similar argument utilizing $L^{2p}$-boundedness of $(d_n)$ and $(Z_n)$ gives
		\begin{eqnarray}\label{C}
			\|C\|_p&\leq&\sum_{n=1}^{N-1}\|(Z_0-Z_n)\otimes d_n\|_p\nonumber\\&\leq&\sum_{n=1}^{N-1}\|Z_0-Z_n\|_{2p}\cdot\|d_n\|_{2p}\nonumber\\&\leq&\sum_{n=1}^{N-1}(\|Z_0\|_{2p}+\|Z_n\|_{2p})\cdot\|d_n\|_{2p}\nonumber\\&\leq&C_p(N-1)\leq C_pN.
		\end{eqnarray}
		
		Expanding the sum in $B$,
		\begin{equation*}
			B=d_0\otimes Z_1+\sum_{n=2}^{N-1} d_{n-1}\otimes Z_n-D_{N-2}\otimes Z_N=:B_1+B_2+B_3.
		\end{equation*}
		Notice that
		\begin{equation}\label{B1}
			\|B_1\|_p\leq\|d_0\|_{2p}\cdot\|Z_1\|_{2p}\leq C_p\leq C_pN,
		\end{equation}
		\begin{equation}\label{B2}
			\|B_2\|_p\leq\sum_{n=2}^{N-1}\|d_{n-1}\otimes Z_n\|_p\leq \sum_{n=2}^{N-1}\|d_{n-1}\|_{2p}\cdot\|Z_n\|_{2p}\leq C_p(N-2)\leq C_pN,
		\end{equation}
		and
		\begin{equation}\label{B31}
			\|B_3\|_p\leq \|D_{N-2}\|_{2p}\|Z_N\|_{2p}\leq C_p\|D_{N-2}\|_{2p}.
		\end{equation}
		By discrete Burkholder--Davis--Gundy inequality \cite[Thm 2.10]{HH80},
		\begin{equation*}
			\E|D_{N-2}|^{2p}\leq C_p\E[(|d_0|^2+\cdots+|d_{N-2}|^2)^p].
		\end{equation*}
		Taking $\frac{1}{2p}$-th power,
		\begin{equation}\label{B32}
			\|D_{N-2}\|_{2p}\leq C_p\||d_0|^2+\cdots+|d_{N-2}|^2\|_p^{1/2}\leq C_p(\||d_0|^2\|_p+\cdots+\||d_{N-2}|^2\|_p)^{1/2}\leq C_p(N-1)^{1/2}.
		\end{equation}
		From \eqref{B31} and \eqref{B32},
		\begin{equation}\label{B3}
			\|B_3\|_p\leq C_p (N-1)^{1/2}\leq C_pN.
		\end{equation}
		Combining \eqref{B1}, \eqref{B2} and \eqref{B3},
		\begin{equation}\label{B}
			\|B\|_p\leq C_pN.
		\end{equation}
		
		Now we deal with $A$, which is the most subtle part. Note that
		\begin{eqnarray*}
			D_n\otimes D_n&=&(D_{n-1}+d_n)\otimes(D_{n-1}+d_n)\\&=&D_{n-1}\otimes D_{n-1}+D_{n-1}\otimes d_n+d_n\otimes D_{n-1}+d_n\otimes d_n,
		\end{eqnarray*}
		i.e.,
		\begin{equation*}
			D_n\otimes D_n-D_{n-1}\otimes D_{n-1}=D_{n-1}\otimes d_n+d_n\otimes D_{n-1}+d_n\otimes d_n.
		\end{equation*}
		Summing from $n=1$ to $n=N-1$,
		\begin{equation}\label{noname20}
			D_{N-1}\otimes D_{N-1}-d_0\otimes d_0=\sum_{n=1}^{N-1}[D_{n-1}\otimes d_n+d_n\otimes D_{n-1}+d_n\otimes d_n].
		\end{equation}
		Set $\Delta_n=D_{n-1}\otimes d_n-d_n\otimes D_{n-1}$. Substitute into \eqref{noname20},
		\begin{eqnarray*}
	&&D_{N-1}\otimes D_{N-1}-d_0\otimes d_0\\&=&\sum_{n=1}^{N-1}[D_{n-1}\otimes d_n+(D_{n-1}\otimes d_n-\Delta_n)+d_n\otimes d_n]\\&=&2\sum_{n=1}^{N-1}D_{n-1}\otimes d_n-\sum_{n=1}^{N-1}\Delta_n+\sum_{n=1}^{N-1}d_n\otimes d_n.
		\end{eqnarray*}
		Rearranging,
		\begin{eqnarray}\label{newA}
			A&=&\sum_{n=1}^{N-1}D_{n-1}\otimes d_n\nonumber\\&=&\frac{1}{2}(D_{N-1}\otimes D_{N-1}-d_0\otimes d_0)-\frac{1}{2}\sum_{n=1}^{N-1}d_n\otimes d_n+\frac{1}{2}\sum_{n=1}^{N-1}\Delta_n\nonumber\\&=:&A_1-A_2+A_3.
		\end{eqnarray}
		Note that left-hand-side of \eqref{newA} is precisely $A$. Utilizing H\"older inequality and \eqref{B32},
		\begin{equation*}
			\|D_{N-1}\otimes D_{N-1}\|_p\leq\|D_{N-1}\|_{2p}\cdot\|D_{N-1}\|_{2p}\leq C_p N^{1/2}\cdot N^{1/2}=C_pN,
		\end{equation*}
		and trivially,
		\begin{equation*}
		 \|d_0\otimes d_0\|_p\leq C_p\leq C_pN,
		\end{equation*}
		and the above two inequalities imply
		\begin{equation}\label{A1}
			\|A_1\|_p\leq C_pN.
		\end{equation}
		Besides, the uniform $L^{2p}$-boundedness of $(d_n)$ gives
		\begin{equation}\label{A2}
			\|A_2\|_p\leq\sum_{n=1}^{N-1}\|d_n\otimes d_n\|_p\leq\sum_{n=1}^{N-1}\|d_n\|_{2p}\|d_n\|_{2p}\leq C_p(N-1)\leq C_pN.
		\end{equation}
		Lastly we deal with $A_3$. Since
		\begin{eqnarray*}
			\E(\Delta_n|\mathcal{G}_{n-1})&=&\E(D_{n-1}\otimes d_n-d_n\otimes D_{n-1}|\mathcal{G}_{n-1})\\&=&D_{n-1}\otimes\E(d_n|\mathcal{G}_{n-1})-\E(d_n|\mathcal{G}_{n-1})\otimes D_{n-1}=0,
		\end{eqnarray*}
		where we use the fact that $(d_n)_n$ is a martingale difference sequence in the last equality. Therefore, $(\Delta_n)_n$ is also a martingale difference, so its partial sum sequence forms a martingale. By discrete Burkholder--Davis--Gundy inequality,
		\begin{equation}\label{BDGforDelta}
			\E|\sum_{n=1}^{N-1}\Delta_n|^p\leq C\E(\sum_{n=1}^{N-1}|\Delta_n|^2)^{p/2}.
		\end{equation}
		Since $\Delta_n=D_{n-1}\otimes d_n-d_n\otimes D_{n-1}$, the elementary inequality $|a\otimes b-b\otimes a|\leq \sqrt{2}|a|\cdot|b|$ implies
		\begin{equation*}
			|\Delta_n|\leq \sqrt{2}|D_{n-1}|\cdot|d_n|,
		\end{equation*}
		so
		\begin{equation*}
			\sum_{n=1}^{N-1}|\Delta_n|^2\leq 2\sum_{n=1}^{N-1}|D_{n-1}|^2|d_n|^2\leq 2\max_{0\leq k\leq N-2}|D_k|^2\sum_{n=1}^{N-1}|d_n|^2.
		\end{equation*}
		Taking $(p/2)$-th power and expectation,
		\begin{equation*}
			\E(\sum_{n=1}^{N-1}|\Delta_n|^2)^{p/2}\leq C_p\E[ \max_{0\leq k\leq N-2}|D_k|^p(\sum_{n=1}^{N-1}|d_n|^2)^{p/2}]. 
		\end{equation*}
	 By H\"older inequality,
	 \begin{equation}\label{noname21}
	 	\E(\sum_{n=1}^{N-1}|\Delta_n|^2)^{p/2}\leq C_p(\E\max_{0\leq k\leq N-2}|D_k|^{2p})^{1/2}\cdot (\E(\sum_{n=1}^{N-1}|d_n|^2)^p)^{1/2}
	 \end{equation}
	Since $(d_n)_n$ is a martingale difference sequence, its partial sum sequence $(D_n)_n$ is a martingale, and by Doob maximal inequality and \eqref{B32},
	\begin{equation}\label{noname22}
		\E\max_{0\leq k\leq N-2}|D_k|^{2p}\leq C_p\E|D_{N-2}|^{2p}\leq C_p N^p.
	\end{equation}
	By Minkowski inequality and $L^{2p}$-boundedness of $(d_n)_n$,
	\begin{equation}\label{noname23}
		\E(\sum_{n=1}^{N-1}|d_n|^2)^p=\|\sum_{n=1}^{N-1}|d_n|^2\|_p^p\leq (\sum_{n=1}^{N-1}\||d_n|^2\|_p)^p\leq C_p(N-1)^p\leq C_p N^p.
	\end{equation}	
		Substitute \eqref{noname22} and \eqref{noname23} into \eqref{noname21}, 
		\begin{equation*}
			\E(\sum_{n=1}^{N-1}|\Delta_n|^2)^{p/2}\leq C_p N^{p/2}\cdot C_pN^{p/2}=C_p N^p.
		\end{equation*}
		Substitute the equality above into \eqref{BDGforDelta},
		\begin{equation*}
			\E|\sum_{n=1}^{N-1}\Delta_n|^p\leq C_p N^p,
		\end{equation*}
		and from the definition of $A_3$, this means precisely 
		\begin{equation}\label{A3}
			\|A_3\|_p\leq C_p N.
		\end{equation}
		Substituting   \eqref{A1},\eqref{A2} and \eqref{A3} into \eqref{newA},
		\begin{equation}\label{A}
			\|A\|_p\leq C_pN.
		\end{equation}
		Substituting \eqref{D},\eqref{C},\eqref{B} and \eqref{A} into \eqref{newI_2},
		\begin{equation}\label{I2}
			\|I_2\|_p\leq CN.
		\end{equation}
		Combining \eqref{noname02} and \eqref{I2}, \eqref{target2} is proved for $N\in\mathbb{N}$.

		For general $N\geq1$, note that
		\begin{eqnarray*}
			&&\int_0^N\int_0^v\xi(u)\otimes\xi(v)dudv\\&=&\int_0^{\lfloor N\rfloor}\int_0^v\xi(u)\otimes\xi(v)dudv+\int_{\lfloor N\rfloor}^N\int_0^{\lfloor N\rfloor}\xi(u)\otimes\xi(v)dudv+\int_{\lfloor N\rfloor}^N\int_{\lfloor N\rfloor}^v\xi(u)\otimes\xi(v)dudv\\&=:&K_1+K_2+K_3.
		\end{eqnarray*}
		From the integer case, $$\E|K_1|^p\leq C\lfloor N\rfloor^p\leq CN^p.$$
		By (i) of Proposition \ref{prop-momentestimateformixing},
		\begin{eqnarray*}
			\|K_2\|_p&=&\Big\|\int_0^{\lfloor N\rfloor}\xi(u)du\otimes\int_{\lfloor N\rfloor}^N\xi(v)dv\Big\|_p\\&\leq&\Big\|\int_0^{\lfloor N\rfloor}\xi(u)du\Big\|_{2p}\cdot\Big\|\int_{\lfloor N\rfloor}^N\xi(v)dv\Big\|_{2p}\\&\leq&(C_pN^{1/2})\cdot\int_{\lfloor N\rfloor}^N\|\xi(v)\|_{2p}dv\\&\leq&C_p N^{1/2}\leq C_p N.
		\end{eqnarray*}
		By stationarity and \eqref{small2},
		\begin{eqnarray*}
			\|K_3\|_p\leq C_p N.
		\end{eqnarray*}
		Combining the estimate for $K_1$, $K_2$ and $K_3$, we see that \eqref{target2} holds for all $N\geq 1$, and the proof is finished. 
	\end{proof}
	We also need to control the following stationary process $\chi$.
	\begin{lemma}
		Set $\chi(t)=\int_0^\infty e^{-u}\eta(t-u)du=\int_{-\infty}^t e^{-(t-u)}\eta(u)du.$ For $0\leq a\leq b<\infty,$ we have
		\begin{equation}\label{chi-1}
			\Big\|\int_a^b\chi(r)dr\Big\|_p\leq C(b-a)^{1/2}.
		\end{equation}
		\begin{equation}\label{chi-2}
			\Big\|\int_a^b\int_a^r\chi(q)\otimes\chi(r)dqdr\Big\|_p\leq C(b-a).
		\end{equation}
	\end{lemma}
	\begin{proof}
		\noindent\textbf{Case 1: $0<b-a\leq 1.$}
		
		By stationarity, 
		\[\Big\|\int_a^b\chi(r)dr\Big\|_p\leq\int_a^b\|\chi(r)\|_pdr\leq C(b-a)\leq C(b-a)^{1/2},\]
		\[\Big\|\int_a^b\int_a^r\chi(q)\otimes\chi(r)dqdr\Big\|_p\leq\int_a^b\int_a^r\|\chi(q)\otimes\chi(r)\|_pdqdr\leq\int_a^b\int_a^r Cdqdr\leq C(b-a)^2\leq C(b-a).\]
		
		\noindent\textbf{Case 2: $b-a> 1.$}
		
		Define
		\[B_r:=\int_a^r\eta(q)dq,\quad P_r=\int_a^r\chi(q)dq,\quad H_r=\chi_r-\chi_a.\]
		
		By Leibniz rule,
		\[\dot{\chi}_r=\frac{d}{dr}[e^{-r}\int_{-\infty}^r e^u\eta(u)du]=-\chi_r+\eta_r.\]
		Therefore,
		\begin{equation}\label{P-B-H-relation}
			H_r=\chi_r-\chi_a=\int_a^r\dot{\chi}_udu=\int_a^r-\chi_u+\eta_udu=-P_r+B_r,
		\end{equation}
		so
		\begin{equation*}
			dH_r=dB_r-dP_r=\eta(r)dr-\chi(r)dr,
		\end{equation*}
		and
		\begin{equation}\label{dB-dH-relation}
			\chi(r)dr=\eta(r)dr-dH_r=dB_r-dH_r.
		\end{equation}
		By stationarity,
		\begin{equation}\label{p-moment-chi}
			\|\chi_r\|_p\leq\int_0^\infty\|e^{-u}\eta(r-u)\|_pdu=\|\eta(0)\|_p\int_0^\infty e^{-u}du\leq C
		\end{equation}
		so we also have 
		\begin{equation}\label{p-moment-H}
			\|H_r\|_p\leq\|\chi_r\|_p+\|\chi_a\|_p\leq C.
		\end{equation}
		By \eqref{P-B-H-relation},
		\[\Big\|\int_a^b\chi(r)dr\Big\|_p=\|P_b\|_p=\|B_b-H_b\|_p\leq\|B_b\|_p+\|H_b\|_p=\Big\|\int_a^b\eta(q)dq\Big\|_p+\|\chi(b)-\chi(a)\|_p.\]
		Applying Proposition \ref{prop-momentestimateformixing},
		\[\Big\|\int_a^b\eta(q)dq\Big\|_p\leq C(b-a)^{1/2}.\]
		By stationarity and \eqref{p-moment-chi},
		\[\|\chi(b)-\chi(a)\|_p\leq \|\chi(b)\|_p+\|\chi(a)\|_p\leq 2\|\chi(0)\|_p\leq C.\]
		Therefore,
		\[\Big\|\int_a^b\chi(r)dr\Big\|_p\leq C(b-a)^{1/2}+C\leq C(b-a)^{1/2}.\]
		Using \eqref{P-B-H-relation} and \eqref{dB-dH-relation} successively and then integrating by parts,
		\begin{equation}\label{chi-tensor}
			\begin{aligned}[b]
				&\quad\int_a^b\int_a^r\chi_q\otimes\chi_rdqdr\\&=\int_a^b\int_a^r\chi_qdq\otimes\chi_rdr\\&=\int_a^bP_r\otimes\chi_rdr\\&=\int_a^bB_r\otimes\chi_rdr-\int_a^bH_r\otimes\chi_rdr\\&=\int_a^b B_r\otimes dB_r-\int_a^b B_r\otimes dH_r-\int_a^b H_r\otimes\chi_rdr\\&=\int_a^b B_r\otimes dB_r-[B_b\otimes H_b-\int_a^b dB_r\otimes H_r]-\int_a^b H_r\otimes \chi_rdr\\&=\int_a^b B_r\otimes \eta_rdr-B_b\otimes H_b+\int_a^b\eta_r\otimes H_rdr-\int_a^b H_r\otimes\chi_rdr
			\end{aligned}
		\end{equation}
		 By Proposition \ref{prop-momentestimateformixing},
\[ \Big\|\int_a^b B_r\otimes\eta_rdr\Big\|_p=\Big\|\int_a^b\int_a^r\eta_q\otimes\eta_rdqdr\Big\|_p\leq C(b-a). \]
By \eqref{p-moment-chi}, \eqref{p-moment-H} and H\"older inequality,		
\[\|B_b\otimes H_b\|_p\leq \|B_b\|_{2p}\|H_b\|_{2p}\leq \int_a^b\|\eta(r)\|_{2p}dr\|H_b\|_{2p}\leq C(b-a),\]
\[ \Big\|\int_a^b\eta_r\otimes H_rdr\Big\|_p\leq\int_a^b\|\eta_r\otimes H_r\|_pdr\leq\int_a^b\|\eta_r\|_{2p}\|H_r\|_{2p}dr\leq C(b-a), \]
and
\[\Big\|\int_a^b H_r\otimes\chi_rdr\Big\|_p\leq\int_a^b\|H_r\otimes \chi_r\|_pdr\leq \int_a^b\|H_r\|_{2p}\|\chi_r\|_{2p}dr\leq C(b-a).\]
Substitute four inequalities above into \eqref{chi-tensor},
\begin{equation*}
	\Big\|\int_a^b\int_a^r\chi_q\otimes\chi_rdqdr\Big\|_{p}\leq C(b-a).
\end{equation*}
		
	\end{proof}

	\noindent\emph{Proof of Proposition \ref{prop-momentestimate}:}
	
Fix $0\leq s<t\leq T.$ We first consider $\xe$.

\noindent\textbf{Case 1: $\ep\geq t-s $.}
Writing $\xe_{s,t}=\int_s^t\ve_rdr$, by H\"older inequality,
\begin{equation*}
	|\xe_{s,t}|^p\leq (t-s)^{p-1}\int_s^t|\ve_r|^pdr.
\end{equation*}
Taking expectation and using Proposition \ref{prop-momentestimateforv},
\begin{equation*}
	\E|\xe_{s,t}|^p\leq (t-s)^{p-1}\int_s^t\E|\ve_r|^pdr\lesssim(t-s)^{p-1}\int_s^t\ep^{-p/2}dr=(t-s)^p\ep^{-p/2}\leq (t-s)^{p/2}.
\end{equation*}

\noindent\textbf{Case 2: $0<\ep\leq t-s$.}

By \eqref{decompowe},
\begin{eqnarray}\label{newxe}
	&&\xe_{s,t}\nonumber\\&=&\frac{1}{\sqrt{\ep}}\int_s^t\we_rdr\nonumber\\&=&\frac{1}{\sqrt{\ep}}\int_s^tz^\ep_rdr+\frac{1}{\sqrt{\ep}}\int_s^tR^\ep_rdr\nonumber\\&=&\frac{1}{\sqrt{\ep}}\int_s^t\int_0^{r/\ep}e^{-u}\eta(\frac{r}{\ep}-u)dudr+\int_s^t\int_0^{r/\ep} e^{-u}F(\xe_{r-\ep u})dudr\nonumber\\&=:&H_1+H_2.
\end{eqnarray}
Since $F$ is bounded,
\begin{equation*}
	|H_2|\leq\|F\|_\infty\int_s^t\int_0^{r/\ep}e^{-u}dudr\lesssim\int_s^t 1dr= t-s,
\end{equation*}
and taking expectation,
\begin{equation}\label{H2}
	\E|H_2|^p\lesssim (t-s)^p\lesssim_T (t-s)^{p/2}
\end{equation}
By change-of-variable,
\begin{equation*}
	H_1=\sqrt{\ep}\int_{s/\ep}^{t/\ep}\int_0^v e^{-(v-w)}\eta(w)dwdv.
\end{equation*}
Set $\xi(v):=\int_0^v e^{-(v-w)}\eta(w)dw$, and recall that $\chi(v):=\int_{-\infty}^v e^{-(v-w)}\eta(w)dw$, then $\xi(v)=\chi(v)-e^{-v}\chi(0),$ and
\begin{equation}\label{newH1}
	H_1=\sqrt{\ep}\int_{s/\ep}^{t/\ep}\chi(v)dv-\sqrt{\ep}\chi(0)\int_{s/\ep}^{t/\ep}e^{-v}dv=:H_{11}-H_{12}.
\end{equation}
Since $\ep\leq t-s$, $\frac{t}{\ep}-\frac{s}{\ep}\geq 1$, so an application of \eqref{chi-1} gives
\begin{equation}\label{H11}
	\E|H_{11}|^p= \ep^{p/2}\E\Big|\int_{s/\ep}^{t/\ep}\chi(v)dv\Big|^p\leq \ep^{p/2}C_p\Big|\frac{t-s}{\ep}\Big|^{p/2}=C_p|t-s|^{p/2}.
\end{equation}
Since $|H_{12}|\leq\sqrt{\ep}|\chi(0)|\int_0^\infty e^{-v}dv=\sqrt{\ep}|\chi(0)|\leq (t-s)^{1/2}|\chi(0)|,$
\begin{equation}\label{H12}
	\E|H_{12}|^p\leq (t-s)^{p/2}\E|\chi(0)|^p\leq C_p(t-s)^{p/2}.
\end{equation}
Substitute \eqref{H11}--\eqref{H12} into \eqref{newH1},
\begin{equation}\label{H1}
	\E|H_1|^p\leq C_p(t-s)^{p/2}.
\end{equation}
Substituting \eqref{H2} and \eqref{H1} into \eqref{newxe},
\begin{equation*}
	\E|\xe_{s,t}|^p\leq C_p (t-s)^{p/2}.
\end{equation*}

Next we consider $\bbxe$.

\noindent\textbf{Case 1: $\ep\geq t-s$.}

By H\"older inequality,
\begin{eqnarray*}
\E|\int_s^t\xe_{s,r}\otimes d\xe_r|^p=\E\Big|\int_s^t\xe_{s,r}\otimes\ve_r dr\Big|^p\leq (t-s)^{p-1}\E\int_s^t|\xe_{s,r}\otimes\ve_r|^pdr.
\end{eqnarray*}
By Proposition \ref{prop-momentestimateforv} and (ii) of Proposition \ref{prop-momentestimate}, 
\begin{eqnarray*}
	\E\int_s^t|\xe_{s,r}\otimes\ve_r|^pdr\leq\int_s^t(\E|\xe_{s,r}|^{2p})^{1/2}(\E|\ve_r|^{2p})^{1/2}dr\lesssim \int_s^t|r-s|^{p/2}\ep^{-p/2}dr\lesssim_p \ep^{-p/2}(t-s)^{(p/2)+1}.
\end{eqnarray*}
Combining two inequalities above,
\begin{equation*}
	\E|\int_s^t\xe_{s,r}\otimes d\xe_r|^p\lesssim \ep^{-p/2}(t-s)^{(3p)/2}\leq (t-s)^p.
\end{equation*}

\noindent\textbf{Case 2: $0<\ep<t-s$.}

By \eqref{decompowe}
\begin{eqnarray}
	\bbxe_{s,t}&=&\int_s^t\xe_{s,r}\otimes d\xe_r\nonumber\\&=&\frac{1}{\ep}\int_s^t\int_s^r\we_q\otimes\we_rdqdr\nonumber\\&=&\frac{1}{\ep}\int_s^t\int_s^r z^\ep_q\otimes z^\ep_rdqdr+\frac{1}{\ep}\int_s^t\int_s^r z^\ep_q\otimes R^\ep_rdqdr\nonumber\\&&+\frac{1}{\ep}\int_s^t\int_s^rR^\ep_q\otimes z^\ep_rdqdr+\frac{1}{\ep}\int_s^t\int_s^rR^\ep_q\otimes R^\ep_rdqdr\nonumber\\&=:&I_1+I_2+I_3+I_4.
\end{eqnarray}
We first deal with $I_2$. Since
\begin{equation*}
	z^\ep_q=\int_0^{q/\ep}e^{-u}\eta(\frac{q}{\ep}-u)du=\int_0^{q/\ep} e^{-(\frac{q}{\ep}-v)}\eta(v)dv=\xi(q/\ep),
\end{equation*}
we have
\begin{eqnarray*}
	I_2&=&\frac{1}{\ep}\int_s^t\int_s^rz^\ep_q\otimes R^\ep_rdqdr\\&=&\frac{1}{\ep}\int_s^t\int_s^r\xi(q/\ep)\otimes \ep^{1/2}\int_0^{r/\ep}e^{-u}F(\xe_{r-\ep u})dudqdr\\&=&\frac{1}{\sqrt{\ep}}\int_s^t\int_s^r\xi(q/\ep)dq\otimes \int_0^{r/\ep}e^{-u}F(\xe_{r-\ep u})dudr.
\end{eqnarray*}
Since $F$ is bounded,
\begin{eqnarray*}
	|I_2|&\leq&\frac{1}{\sqrt{\ep}}\int_s^t\Big|\int_s^r\xi(q/\ep)dq\Big|\cdot \Big|\int_0^{r/\ep}e^{-u}F(\xe_{r-\ep u})du\Big|dr\\&\lesssim&\frac{1}{\sqrt{\ep}}\int_s^t\Big|\int_s^r\xi(q/\ep)dq\Big|dr\\&=&\ep^{3/2}\int_{s/\ep}^{t/\ep}\Big|\int_{s/\ep}^r\xi(q)dq\Big|dr,
\end{eqnarray*}
where we made change-of-variable $q/\ep\mapsto q$ and then $r/\ep\mapsto r$ in the last equality, so
\begin{equation}\label{newI2}
	\|I_2\|_p\lesssim \ep^{3/2}\Big\|\int_{s/\ep}^{t/\ep}\Big|\int_{s/\ep}^r\xi(q)dq\Big|dr\Big\|_p\leq \ep^{3/2}\int_{s/\ep}^{t/\ep} \Big\|\int_{s/\ep}^r\xi(q)dq\Big\|_pdr.
\end{equation}
But
\begin{eqnarray*}
	&&\Big\|\int_{s/\ep}^r\xi(q)dq\Big\|_p\\&=&\Big\|\int_{s/\ep}^r\chi(q)-e^{-q}\chi(0)dq\Big\|_p\\&\leq&\Big\|\int_{s/\ep}^r\chi(q)dq\Big\|_p+\Big\|\int_{s/\ep}^re^{-q}\chi(0)dq\Big\|_p\\&\leq&C_p(r-\frac{s}{\ep})^{1/2}+C_p\int_{s/\ep}^r e^{-q}dq\\&\leq& C_p(r-\frac{s}{\ep})^{1/2}+C_p,
\end{eqnarray*}
where we apply \eqref{chi-1} in the second-last step. Substitute into \eqref{newI2},
\begin{equation*}
	\|I_2\|_p\leq \ep^{3/2}\int_{s/\ep}^{t/\ep} C_p[(r-\frac{s}{\ep})^{1/2}+1]dr\lesssim_p \ep^{3/2}(\frac{t-s}{\ep})^{3/2}+\ep^{3/2}(\frac{t-s}{\ep})=(t-s)^{3/2}+\ep^{1/2}(t-s).
\end{equation*}
Recall the assumption $\ep\leq t-s$, so that
\begin{equation*}
	\|I_2\|_p\leq C_p (t-s)^{3/2}\leq C_{p,T} (t-s).
\end{equation*}
By the same argument
\begin{equation*}
	\|I_3\|_p\leq C_{p,T}(t-s).
\end{equation*}
Plainly, \eqref{esforre} implies
\begin{equation*}
	|I_4|\leq\frac{1}{\ep}\int_s^t\int_s^r|R^\ep_q|\cdot|R^\ep_r|dqdr\lesssim\int_s^t\int_s^r dqdr=\frac{(t-s)^2}{2}\lesssim_T t-s.
\end{equation*}
Taking expectation,
\begin{equation*}
	\|I_4\|_p\leq C_T(t-s).
\end{equation*}
Lastly we deal with $I_1$. By change-of-variable,
\begin{eqnarray*}
	&&I_1\nonumber\\&=&\frac{1}{\ep}\int_s^t\int_s^r[\int_0^{q/\ep}e^{-u}\eta(\frac{q}{\ep}-u)du\otimes\int_0^{r/\ep}e^{-v}\eta(\frac{r}{\ep}-v)dv]dqdr\nonumber\\&=&\ep\int_{s/\ep}^{t/\ep}\int_{s/\ep}^r[\int_0^q e^{-u}\eta(q-u)du\otimes \int_0^r e^{-v}\eta(r-v)dv]dqdr\nonumber\\&=&\ep\int_{s/\ep}^{t/\ep}\int_{s/\ep}^r[\int_0^q e^{-(q-u)}\eta(u)du\otimes \int_0^r e^{-(r-v)}\eta(v)dv]dqdr\nonumber\\&=:&\ep\int_{s/\ep}^{t/\ep}\int_{s/\ep}^r \xi(q)\otimes\xi(r)dqdr,
\end{eqnarray*}
where we recall that $\xi(q)=\int_0^q e^{-(q-u)}\eta(u)du,$ and $\chi(q):=\int_{-\infty}^q e^{-(q-u)}\eta(u)du.$
Set
\begin{equation}\label{newI1}
	I_1^*:=\ep\int_{s/\ep}^{t/\ep}\int_{s/\ep}^r \chi(q)\otimes\chi(r)dqdr.
\end{equation}
   Since $\xi(q)=\chi(q)-e^{-q}\chi(0),$ we have
   \begin{eqnarray}\label{difference}
   	&&I_1-I_1^*\nonumber\\&=&-\ep\int_{s/\ep}^{t/\ep}\int_{s/\ep}^r\chi(q)\otimes[e^{-r}\chi(0)]dqdr-\ep\int_{s/\ep}^{t/\ep}\int_{s/\ep}^r e^{-q}\chi(0)\otimes\chi(r)dqdr\nonumber\\&&+\ep\int_{s/\ep}^{t/\ep}\int_{s/\ep}^r e^{-q}\chi(0)\otimes e^{-r}\chi(0)dqdr=:T_1+T_2+T_3.
   \end{eqnarray}
   Since
   \begin{equation*}
   	T_1=-\ep\int_{s/\ep}^{t/\ep}e^{-r}\int_{s/\ep}^r\chi(q)dqdr\otimes\chi(0),
   \end{equation*}
   we have
   \begin{equation*}
   	|T_1|\leq\ep\int_{s/\ep}^{t/\ep}e^{-r}\Big|\int_{s/\ep}^r\chi(q)dq\Big|dr\cdot|\chi(0)|.
   \end{equation*}
   Taking $L^p$-norm, applying H\"older inequality, \eqref{chi-1}, and recalling $\ep<t-s$,
   \begin{eqnarray}\label{T1}
   	&&\|T_1\|_p\nonumber\\&\leq&\ep\int_{s/\ep}^{t/\ep} e^{-r}\Big\|\int_{s/\ep}^{r}\chi(q)dq\Big\|_{2p}dr\cdot\|\chi(0)\|_{2p}\nonumber\\&\leq&C_p\ep\int_{s/\ep}^{t/\ep} e^{-r}(r-\frac{s}{\ep})^{1/2}dr\nonumber\\&=&C_p\ep e^{-s/\ep}\int_0^{(t-s)/\ep}r^{1/2}e^{-r}dr\nonumber\\&\leq&C_p\ep\leq C_p (t-s).
   \end{eqnarray}
   Direct computation gives
   \begin{eqnarray*}
   	&&T_2\nonumber\\&=&-\ep e^{-s/\ep}\chi(0)\otimes\int_{s/\ep}^{t/\ep}\chi(r)dr+\ep\chi(0)\otimes\int_{s/\ep}^{t/\ep}\chi(r)e^{-r}dr\nonumber\\&=:&T_{21}+T_{22}.
   \end{eqnarray*}
   By H\"older inequality and \eqref{chi-1},
   \begin{eqnarray}\label{T21}
   	\|T_{21}\|_p&\leq&\ep e^{-s/\ep}\|\chi(0)\|_{2p}\cdot\Big\|\int_{s/\ep}^{t/\ep}\chi(r)dr\Big\|_{2p}\nonumber\\&\leq&C_p\ep e^{-s/\ep}(\frac{t-s}{\ep})^{1/2}\nonumber\\&\leq&C_p\ep^{1/2}e^{-s/\ep}(t-s)^{1/2}\leq C_p(t-s).
   \end{eqnarray}
   By H\"older inequality, triangular inequality and stationarity,
   \begin{eqnarray}\label{T22}
   	\|T_{22}\|_p\leq\ep\|\chi(0)\|_{2p}\cdot\int_{s/\ep}^{t/\ep}\|\chi(0)\|_{2p}e^{-r}dr\leq C_p\ep\leq C_p(t-s).
   \end{eqnarray}
   \eqref{T21} and \eqref{T22} imply
   \begin{equation}\label{T2}
   	\|T_2\|_p\leq C_p(t-s)
   \end{equation}
\begin{eqnarray*}
&&T_3\\&=&\ep\int_{s/\ep}^{t/\ep}\int_{s/\ep}^r e^{-q}\chi(0)\otimes e^{-r}\chi(0)dqdr\\&=&\ep\int_{s/\ep}^{t/\ep}\int_{s/\ep}^re^{-q}e^{-r}dqdr\chi(0)\otimes\chi(0)\\&=&[\ep e^{-\frac{s}{\ep}}(e^{-\frac{s}{\ep}}-e^{-\frac{t}{\ep}})-\frac{\ep}{2}(e^{-\frac{2s}{\ep}}-e^{-\frac{2t}{\ep}})]\chi(0)\otimes\chi(0),
\end{eqnarray*}
but
$$|\ep e^{-\frac{s}{\ep}}(e^{-\frac{s}{\ep}}-e^{-\frac{t}{\ep}})|\leq\ep e^{-\frac{s}{\ep}}|\frac{t-s}{\ep}|\leq|t-s|,$$
$$|\frac{\ep}{2}(e^{-\frac{2s}{\ep}}-e^{-\frac{2t}{\ep}})|\leq\frac{\ep}{2}|\frac{2(t-s)}{\ep}|\leq|t-s|,$$
so
\begin{equation}\label{T3}
	\E|T_3|^p\lesssim|t-s|^p\E|\chi(0)\otimes\chi(0)|^p\leq C_p|t-s|^p.
\end{equation}
Combining \eqref{difference} \eqref{T1}, \eqref{T2} and \eqref{T3}, 
\[\|I_1-I_1^*\|_p\leq C|t-s|.\]
This implies that, to show $\E|I_1|^p\leq C_p|t-s|^p$, it suffices to show $\E|I_1^*|^p\leq C_p|t-s|^p.$ Indeed, by \eqref{chi-2}  and \eqref{newI1},
\begin{eqnarray*}
	&&\E|I_1^*|^p\\&\leq&\ep^p\E\Big|\int_{s/\ep}^{t/\ep}\int_{s/\ep}^r \chi(q)\otimes\chi(r)dqdr\Big|^p\\&\leq&\ep^pC_p(\frac{t-s}{\ep})^p=C_p(t-s)^p.
\end{eqnarray*}
This finishes the proof for $\bfxe.$ The proof for $\bfue$ is finished by the same and easier argument, by noting that
\begin{eqnarray*}
	\bbue_{s,t}&=&\int_s^t\ue_{s,r}\otimes d\ue_r\nonumber\\&=&\frac{1}{\ep}\int_s^t\int_s^r\eta(u/\ep)\otimes\eta(r/\ep)dudr\nonumber\\&=&\ep\int_{s/\ep}^{t/\ep}\int_{s/\ep}^{r}\eta(u)\otimes\eta(r)dudr.
\end{eqnarray*}		 
	\qed
	
	As a corollary of Kolmogorov tightness criterion for rough paths \cite[Theorem 3.10]{FH20}, we have
	\begin{coro}\label{coro-tight} For each $1/3<\alpha<1/2,$
		$(\bfue)_{0<\ep\leq 1}$ and $(\bfxe)_{0<\ep\leq 1}$ are tight in $\mathscr{C}^{0,\al}([0,T];\mathbb{R}^d)$.
	\end{coro}

	\section{Iterated Weak Invariance Principle for $\bfue$}\label{section-limitofU}
	
	  Now we identify $\mathbf{U}$. Set
	  \begin{equation*}
	  	D_t=\int_0^\infty \E[\eta(t+u)\mid\HH_t]du,\quad L_t:=\intot\eta(r)dr+D_t-D_0,\quad D^\ep_t:=D_{t/\ep}, \quad M^\ep_t:=\sqrt{\ep}L_{t/\ep},
	  \end{equation*}
	  and
	  \begin{equation*}
	  	H^{\ep,ij}_t:=W^{\ep,i}_t(W^{\ep,j}_t+D^{\ep,j}_t).
	  \end{equation*}
	  A change-of-variable gives
	  \begin{equation}\label{M-U-D-relation}
	  	\me_t=\sqrt{\ep}\Big(\int_0^{t/\ep}\eta(r)dr+D_{t/\ep}-D_0\Big)=\frac{1}{\sqrt{\ep}}\intot\eta(\frac{r}{\ep})dr+\sqrt{\ep}(D^\ep_t-D_0)=U^\ep_t+\sqrt{\ep}(D^\ep_t-D_0).
	  \end{equation}
	  
	  For
	  $(x,X)\in\mathbb R^d\times\mathbb R^{d\times d}$ and a fixed $\vp\in C_c^\infty(\mathbb{R}^{d}\times\mathbb{R}^{d\times d}),$ define
	  \begin{align*}
	  	a_j(x,X)&:=\partial_{x_j}\varphi(x,X)
	  	+\sum_k x^k\partial_{X_{kj}}\varphi(x,X),\\
	  	\Phi(x,X,y)&:=\sum_j a_j(x,X)y^j,\\
	  	\Theta^{ij}(x,X)&:=\partial_{x_i}a_j(x,X)
	  	+\sum_kx^k\partial_{X_{ki}}a_j(x,X).
	  \end{align*}
	  Equivalently,
	  \begin{align*}
	  	\Theta^{ij}={}&\partial^2_{x_ix_j}\varphi
	  	+\sum_kx^k\partial^2_{x_iX_{kj}}\varphi
	  	+\sum_kx^k\partial^2_{X_{ki}x_j}\varphi\\
	  	&+\sum_{k,l}x^kx^l\partial^2_{X_{ki}X_{lj}}\varphi
	  	+\partial_{X_{ij}}\varphi.
	  \end{align*}
	  For each $Y\in C^1$ with $Y_0=0$ and $\bfy=(Y,\by)$ with $\by_{s,t}:=\int_s^t Y_{s,r}\otimes dY_r,$ by chain rule,
	  \begin{equation}\label{chain}
	  	\frac{d}{dt}\vp(\bfy_t)=\sum_j a_j(\bfy_t)\dot{Y}^j_t, \quad \frac{d}{dt}a_j(\bfy_t)=\sum_i[\Theta^{ij}(\bfy_t)\dot{Y}^i_t]
	  \end{equation}

\begin{lemma}\label{lemma-D-L} (i) For each $p>1,$
\[	\sup_{t\geq 0}\|D_t\|_p\leq C_{p}. \]

(ii) $L$ is a square-integrable $(\HH_t)_{t\geq 0}$ martingale (so $\me$ is a $(\HH_{t/\ep})_{t\geq 0}$ martingale).

(iii)
\begin{equation}\label{ex-etaD}
	\E [\eta^i(t)D^j_t]=\int_0^\infty R^{ij}(u)du= A^{ij}.
\end{equation}

\end{lemma}
\begin{proof}
	For each $t\geq 0,$ $u>0,$ by \eqref{augment-v.s.-usual},
	\[\E(\eta(t+u)\mid\HH_t)=\E(\eta(t+u)\mid \hat\F_{-\infty}^t)=\E(\eta(t+u)\mid \F_{-\infty}^t).\]
	Since conditional expectation is a contraction on $L^p,$ we have by \eqref{continuous-time-case} that
	\[\|\E(\eta(t+u)\mid\HH_t)\|_p= \|\E(\eta(t+u)\mid\F_{-\infty}^{t}\|_p\leq C_pm(u)^{1/2p}\|\eta(t+u)\|_{2p}. \]
	By stationarity of $\eta,$
	\[\|\E(\eta(t+u)\mid\HH_t)\|_p\leq C_pm(u)^{1/2p}\|\eta(0)\|_{2p}\leq C_pe^{-c_pu},\]
	so 
	\begin{equation}\label{p-moment-D}
		\|D_t\|_p\leq\int_0^\infty\|\E(\eta(t+u)\mid\HH_t)\|_pdu\leq \int_0^\infty C_pe^{-c_pu}du\leq C_p.
	\end{equation}
	By stationarity and \eqref{p-moment-D},
	\[\|L_t\|_2\leq\int_0^t\|\eta(r)\|_2dr+\|D_t\|_2+\|D_0\|_2\leq C_t, \] so $L$ is square integrable. By conditional Fubini theorem and tower property, for $t\geq s\geq 0,$
	\begin{align*}
		\E(D_t\mid\HH_s)&=\int_0^\infty \E(\eta(t+u)\mid \HH_s)du\\&=\int_t^\infty \E(\eta(r)\mid\HH_s)dr=D_s-\int_s^t \E(\eta(r)\mid\HH_s)dr,
	\end{align*}
	from which we obtain
	\[\E(L_t-L_s\mid\HH_s)=0,\] so $L$ is a martingale.
	Besides,
	\begin{equation*}
		\begin{aligned}[b]		
			\E [\eta^i(t)D^j_t]&=\E [\eta^i(t)\int_0^\infty \E(\eta^j(t+u)\mid\HH_t)du]\\&=\E\int_0^\infty\E(\eta^i(t)\eta^j(t+u)\mid\HH_t)du\\&=\int_0^\infty \E(\eta^i(t)\eta^j(t+u))du=\int_0^\infty R^{ij}(u)du=A^{ij}.
		\end{aligned}
	\end{equation*}
\begin{remark}\label{cadlag-modification}
	Since $(\HH_t)_t$ satisfies the usual condition, $L$ has a c\`adl\`ag modification $\tilde{L}$ on $(0,\infty).$  We also define for each $n\in\mathbb{N}$
	\[L^{(-n)}_t:=\int_{-n}^t \eta(r)dr+D_t-D_{-n},\quad t\in (-\infty,\infty),\] and by the same argument leading to Lemma \ref{lemma-D-L} (ii), $L^{(-n)}$ is a martingale on $[-n,\infty),$ set
	\[\tilde{D}^{(-n)}_t:=\tilde{L}^{(-n)}_t-\int_{-n}^t\eta(r)dr+D_{-n},\quad t\geq -n. \]
	For each $n,m\in\mathbb{N}$ with $n\geq m,$ $\tilde{D}^{(-n)}$ and $\tilde{D}^{(-m)}$ are both c\`adl\`ag modification of $D$ on $(-m,\infty),$ so they are indistinguishable. Letting $n\to\infty,$ $D$ has a c\`adl\`ag modification $\tilde{D}$ on $(-\infty,\infty).$ Below we always use c\`adl\`ag modification $\tilde{L}$ and $\tilde{D}$ of $L$ and $D,$ and still denote them as $L$ and $D.$
	
\end{remark}
\end{proof}

\noindent\emph{Proof of Theorem \ref{main1}:}

By \eqref{chain} and then \eqref{M-U-D-relation},
\[d\vp(\bfue_t)=\sum_j a_j(\bfue_t)dU^{\ep,j}_t=\sum_j a_j(\bfue_t)dM^{\ep,j}_t-\sqrt{\ep}\sum_j a_j(\bfue_t)dD^{\ep,j}_t,\]
so
\begin{equation*}
	\vp(\bfue_t)-\vp(0)=\sum_j\intot a_j(\bfue_r)dM^{\ep,j}_r-\sqrt{\ep}\sum_j \intot a_j(\bfue_r)dD^{\ep,j}_r.
\end{equation*}
By integration-by-part and \eqref{chain},
\begin{equation*}
	\begin{aligned}[b]
		\intot a_j(\bfue_r)dD^{\ep,j}_r&=a_j(\bfue_t)D^{\ep,j}_t-a_j(0)D^{\ep,j}_0-\intot D^{\ep,j}_r da_j(\bfue_r)\\&=a_j(\bfue_t)D^{\ep,j}_t-a_j(0)D^{\ep,j}_0-\sum_i\intot D^{\ep,j}_r\Theta^{ij}(\bfue_r)dU^{\ep,i}_r\\&=a_j(\bfue_t)D^{\ep,j}_t-a_j(0)D^{\ep,j}_0-\frac{1}{\sqrt{\ep}}\sum_i\intot D^{\ep,j}_r\Theta^{ij}(\bfue_r)\eta^i(\frac{r}{\ep})dr
	\end{aligned}
\end{equation*}
Combining two equations and rearranging,
\begin{align}\label{U-martingale}
	N^{\ep,U}_t:=&\varphi(\bfue_t)-\varphi(0)
	+\sqrt\ep\sum_j\bigl[
	a_j(\mathbf U_t^\ep)D_t^{\ep,j}-a_j(0)D_0^j\bigr]
	\nonumber\\&-\int_0^t\sum_{i,j}\Theta^{ij}(\mathbf U_r^\ep)
	\eta^i(r/\ep)D_{r/\ep}^j\,dr
	=\sum_j\int_0^t a_j(\mathbf U_r^\ep)\,dM_r^{\ep,j}.
\end{align}
Since $M^\ep$ is a martingale and $a_j$ is bounded, the right-hand-side of \eqref{U-martingale} is a martingale.
By \eqref{p-moment-D}, for each $0\leq t\leq T,$
\begin{equation}\label{approx-U-martingale-1}
	\sqrt\varepsilon\E|
	a_j(\mathbf U_t^\ep)D_t^{\ep,j}-a_j(0)D_0^j|\to 0.
\end{equation}
Note that
\begin{equation}
	\intot\Theta^{ij}(\bfue_r)\eta^i(r/\ep)D^j_{r/\ep}dr-\intot\Theta^{ij}(\bfue_r)A^{ij}dr=\intot\Theta^{ij}(\bfue_r)q^{ij}_{r/\ep}dr,
\end{equation}
where $q^{ij}_r:=\eta^i(r)D^j_r-A^{ij},$
and $q^{ij}$ is stationary and ergodic with
\[\E q^{ij}_0=\E\eta^i(0)D^j_0-A^{ij}=0,\]
the equality coming from \eqref{ex-etaD}. It is straightforward to check that $(\eta,D)$ is jointly strictly stationary and ergodic, so by Lemma \ref{lemma-vanishing-weight},
\begin{equation}\label{approx-U-martingale-2}
	\lim\limits_{\ep\to0}\E\supt\Big|\intot\Theta^{ij}(\bfue_r)(\eta^i(r/\ep)D^j_{r/\ep}-A^{ij})dr\Big|=0,
\end{equation}
Set 
\[e^{\ep,U}_t:=-\sqrt\varepsilon\sum_j[
a_j(\mathbf U_t^\ep)D_t^{\ep,j}-a_j(0)D_0^j]+\sum_{ij}\intot\Theta^{ij}(\bfue_r)(\eta^i(r/\ep)D^j_{r/\ep}-A^{ij})dr,\]
so that
\begin{equation*}
	M^{\ep,U}_t:=N^{\ep,U}_t+e^{\ep,U}_t=\vp(\bfue_t)-\vp(0)-\sum_{i,j}\intot\Theta^{ij}(\bfue_r)A^{ij}dr.
\end{equation*}
Set \begin{equation}\label{def-Lu}
	\LL_U\vp(x,X):=\sum_{i,j}A^{ij}\Theta^{ij}(x,X),
\end{equation}
so that the equation above is written as
\begin{equation*}
	M^{\ep,U}_t=\vp(\bfue_t)-\vp(0)-\intot\LL_U\vp(\bfue_r)dr
\end{equation*}
Let $t>s\geq t_1,...,t_k\geq 0,$ $f\in C_b(\mathbb{R}^{k(d+d^2)}),$ then
\begin{equation}\label{Meu-martingale}
	\E[(M^{\ep,U}_t-M^{\ep,U}_s)f(\bfue_{t_1},...,\bfue_{t_k})]=\E[(N^{\ep,U}_t-N^{\ep,U}_s)f(\bfue_{t_1},...,\bfue_{t_k})]+\E[(e^{\ep,U}_t-e^{\ep,U}_s)f(\bfue_{t_1},...,\bfue_{t_k})].
\end{equation}
The martingale property of $N^{\ep,U}$ gives
\begin{equation}\label{Meu-martingale-1}
	\E[(N^{\ep,U}_t-N^{\ep,U}_s)f(\bfue_{t_1},...,\bfue_{t_k})]=0.
\end{equation}
On the other hand, by \eqref{approx-U-martingale-1} and \eqref{approx-U-martingale-2}, we have for each $t\geq 0,$
\[\E |e^{\ep,U}_t|\to 0,\quad\ep\to0,\] so
\[|\E[(e^{\ep,U}_t-e^{\ep,U}_s)f(\bfue_{t_1},...,\bfue_{t_k})]|\leq\|f\|_\infty \E|e^{\ep,U}_t-e^{\ep,U}_s|\to 0.\]
Combined with \eqref{Meu-martingale}--\eqref{Meu-martingale-1}, we have
\begin{equation}\label{limit-MeU}
	\lim\limits_{\ep\to 0}\E[(M^{\ep,U}_t-M^{\ep,U}_s)f(\bfue_{t_1},...,\bfue_{t_k})]=0
\end{equation}
For each $\bfy\in\mathscr{C}^{0,\al},$ set
\begin{align*}
	f_{s,t}^{t_1,...,t_k}(\bfy):=&\{[\vp(\bfy_t)-\vp(\bfy_0)-\intot\LL_U\vp(\bfy_r)dr]\\&-[\vp(\bfy_s)-\vp(\bfy_0)-\intos\LL_U\vp(\bfy_r)dr]\}\cdot f(\bfy_{t_1},...,\bfy_{t_k}).
\end{align*}
It is straightforward to check that $f_{s,t}^{t_1,...,t_k}\in C_b(\mathscr{C}^{0,\al},\mathbb{R}).$
Since $(\law(\bfue))_{0<\ep\leq 1}$ is tight in $\mathscr{C}^{0,\al},$ there is subsequence (we do not change the superscript $\ep$) such that $(\law(\bfue))_{0<\ep\leq 1}$ converges  weakly to some probability measure $\mu$ on $\mathscr{C}^{0,\al}.$ Suppose that $\mathbf U$ is a realization of $\mu$ on some probability space $(\bar\Omega,\bar\F,\bar P),$ and we denote $\bar\E$ as the expectation on $(\bar\Omega,\bar\F,\bar P),$ then the definition of weak convergence gives
\[\E f_{s,t}^{t_1,...,t_k}(\bfue)\to \bar\E f_{s,t}^{t_1,...,t_k}(\mathbf{U}).   \] On the other hand, \eqref{limit-MeU} means precisely
\[\E f_{s,t}^{t_1,...,t_k}(\bfue)\to 0. \]
Equating  right-hand-sides of the two equations above, 
\[\bar\E f_{s,t}^{t_1,...,t_k}(\mathbf{U})=0.\] By monotone class theorem, for each bounded $\sigma(\mathbf{U}_r:0\leq r\leq s)$ measurable function $g,$
\begin{equation}\label{confirm-martingale-U}
	\bar\E[\{[\vp(\mathbf U_t)-\vp(\mathbf U_0)-\intot\LL_U\vp(\mathbf U_r)dr]-[\vp(\mathbf U_s)-\vp(\mathbf U_0)-\intos\LL_U\vp(\mathbf U_r)dr]\}\cdot g]=0.
\end{equation} Therefore,
\[t\mapsto \vp(\mathbf U_t)-\vp(\mathbf U_0)-\intot\LL_U\vp(\mathbf U_r)dr\] is a $\F^\mathbf{U}$ martingale.
Set
\begin{equation}\label{phi1phi2}
	\vp_1^i(x,X):=x_i,\quad\vp^{ij}_2(x,X):=X_{ij}.
\end{equation}
A direct computation gives
\[\LL_U\vp_1^i(x,X)=0,\quad\LL_U\vp_2^{ij}(x,X)=A^{ij}.\]
Although $\vp_1$ and $\vp_2$ are not globally bounded, a localization argument yields that
\begin{equation*}
	M^{U,i}_t:=\vp_1^i(\mathbf U_t)-\vp_1^i(\mathbf U_0)-\intot\LL_U\vp_1^i(\mathbf U_r)dr=U^i_t
\end{equation*}
and
\begin{equation*}
	N^{U,ij}_t:=\vp_2^{ij}(\mathbf U_t)-\vp_2^{ij}(\mathbf U_0)-\intot\LL_U\vp_2^{ij}(\mathbf U_r)dr=\mathbb U^{ij}_t-A^{ij}t
\end{equation*}
are two continuous local martingales. By Remark \ref{remark-VL} and \eqref{qv},
\begin{align}\label{qv-U}
	[U^i,U^j]_t=[M^{U,i},M^{U,j}]_t&=(A^{ij}+A^{ji})t=S^{ij}t,\nonumber\\  [M^{U,k},N^{U,ij}]_t&=\intot S^{kj}U^i_rdr,\nonumber\\ [N^{U,ij},N^{U,kl}]_t&=\intot S^{jl}U^i_rU^k_rdr. 
\end{align}
By L\'evy characterization theorem, 
$U_t=S^{1/2}B_t$ for some standard Brownian motion $B$ on a possibly enlarged probability space.
Set \[I^{U,ij}_t:=\int_0^t U^i_rdU^j_r,\] so that by \eqref{qv-U},
\[[I^{U,ij}]_t=\intot|U^i_r|^2d[U^j]_r=\intot(U^i_r)^2S^{jj}dr,\] and
\begin{equation*}
	[N^{U,ij},I^{U,ij}]_t=\intot U^i_rd[N^{U,ij},M^{U,j}]_r=\intot (U^i_r)^2S^{jj}dr.
\end{equation*}
By \eqref{qv-U} again,
\begin{equation*}
	[N^{U,ij}-I^{U,ij}]_t=[N^{U,ij}]_t+[I^{U,ij}]_t-2[N^{U,ij},I^{U,ij}]_t=0.
\end{equation*}
Therefore, $N^{U,ij}=I^{U,ij},$
and 
\[\mathbb U^{ij}_t=N^{U,ij}_t+A^{ij}t=\intot U^i_rdU^j_r+A^{ij}t,\]
and a direct computation using Chen relation gives
\[\mathbb U^{ij}_{s,t}=\int_s^t U^i_{s,r}\circ dU^j_r+\frac{A^{ij}-A^{ji}}{2}(t-s).\] The law of $(U,\mathbb U)$ is then uniquely determined. Therefore, every converging subsequence of $\bfue$ has the same limit in the sense of distribution. The desired weak convergence follows from tightness of $(\bfue)_{0<\ep\leq 1}.$
 \qed
\section{Smoluchowski--Kramers Approximation}\label{section-limitofX}
In this section we prove Theorem \ref{main2} and Theorem \ref{main3}. To this aim, we establish the following result first.
\begin{lemma}
	For each fixed $t>0,$
	\[\lim_{\ep\to 0}|\E H^{\ep,ij}_t-\bar{\Gamma}^{ij}|=0,\]
	and
	\[\int_0^T|\E H^{\ep,ij}_t-\bar{\Gamma}^{ij}|dt\leq C_T\sqrt{\ep}.\]
\end{lemma}
\begin{proof}
Recall from \eqref{decompowe},
\begin{equation*}
	\we_s=z^\ep_s+R^\ep_s.
\end{equation*}		
with
\[|R^\ep_s|=\Big|\frac{1}{\sqrt{\ep}}\intos e^{-\epi(s-r)}F(\xe_r)dr\Big|\leq \sqrt{\ep}\|F\|_\infty,   \]
and
\[z^\ep_s=\frac{1}{\ep}\int_0^s e^{-\epi(s-r)}\eta(r/\ep)dr.\]
A change-of-variable gives
\begin{equation*}
	\begin{aligned}[b]
	z^\ep_s&=\int_0^{s/\ep}e^{-r}\eta(\frac{s}{\ep}-r)dr\\&=\int_0^\infty e^{-r}\eta(\frac{s}{\ep}-r)dr-\int_{s/\ep}^\infty e^{-r}\eta(\frac{s}{\ep}-r)dr\\&=\chi(s/\ep)-e^{-s/\ep}\chi(0).
	\end{aligned}
\end{equation*}
Combining equations above, we obtain
\[\we_s=\chi(s/\ep)-e^{-s/\ep}\chi(0)+R^\ep_s\quad\text{with}\quad\|R^\ep\|_\infty\leq\sqrt{\ep}\|F\|_\infty.\]
Therefore, setting $E^\ep_t:=\we_t-\chi(t/\ep),$ we have by stationarity that,
\begin{equation}
	\|E^\ep_t\|_p\leq C(e^{-t/\ep}+\sqrt{\ep}).
\end{equation}	
Note that
\begin{align*}
	&\quad\|H^{\ep,ij}_t-\chi^i(t/\ep)(\chi^j(t/\ep)+D^j_{t/\ep})\|_p\\&= \|W^{\ep,i}_t(W^{\ep,j}_t+D^{\ep,j}_t)-\chi^i(t/\ep)(\chi^j(t/\ep)+D^j_{t/\ep})\|_p\\&=\|(E^{\ep,i}_t+\chi^i(t/\ep))(W^{\ep,j}_t+D^{\ep,j}_t)-\chi^i(t/\ep)\chi^j(t/\ep)-\chi^i(t/\ep)D^j_{t/\ep}\|_p\\&=\|E^{\ep,i}_tW^{\ep,j}_t+E^{\ep,i}_tD^{\ep,j}_t+\chi^i(t/\ep)E^{\ep,j}_t\|_p,
\end{align*} 
and that by H\"older inequality, Proposition \ref{prop-momentestimateforv}, and \eqref{p-moment-D},
\[\|E^{\ep,i}_tW^{\ep,j}_t\|_p\leq \|E^{\ep,i}_t\|_{2p}\|W^{\ep,j}_t\|_{2p}\leq C(e^{-t/\ep}+\sqrt{\ep}), \]
\[\|E^{\ep,i}_tD^{\ep,j}_t\|_p\leq \|E^{\ep,i}_t\|_{2p}\|D^{\ep,j}_t\|_{2p}\leq C(e^{-t/\ep}+\sqrt{\ep}),\]
\[\|\chi^i(t/\ep)E^{\ep,j}_t\|_{p}\leq\|\chi^i(t/\ep)\|_{2p}\|E^{\ep,j}_t\|_{2p}\leq C(e^{-t/\ep}+\sqrt{\ep}).\]
Combining four equations above,
\begin{equation}\label{H-approx-p}
	\|H^{\ep,ij}_t-\chi^i(t/\ep)(\chi^j(t/\ep)+D^j_{t/\ep})\|_p\leq C(e^{-t/\ep}+\sqrt{\ep}).
\end{equation}
Set $\tilde{H}^{\ep,ij}_t:=\chi^i(t/\ep)(\chi^j(t/\ep)+D^j_{t/\ep}).$
By stationarity,
\begin{align*}
	\E\chi^i(t/\ep)\chi^j(t/\ep)&=\E\chi^i(0)\chi^j(0)\\&=\E\int_0^\infty\int_0^\infty e^{-u-v}\eta^i(-u)\eta^j(-v)dudv\\&=\int_0^\infty\int_0^\infty e^{-u-v}R^{ij}(u-v)dudv=\Gamma^{ij}.
\end{align*}
By stationarity and total expectation formula,
\begin{align*}
	\E\chi^i(t/\ep)D^j_{t/\ep}&=\E\chi^i(0)D^j_0\\&=\E[\chi^i(0)\int_0^\infty\E(\eta^j(v)\mid\HH_0)dv]\\&=\E\int_0^\infty\E(\chi^i(0)\eta^j(v)\mid\HH_0)dv\\&=\int_0^\infty\E[\chi^i(0)\eta^j(v)]dv\\&=\int_0^\infty\int_0^\infty e^{-u}\E[\eta^i(-u)\eta^j(v)]dudv\\&=\int_0^\infty\int_0^\infty e^{-u}R^{ij}(u+v)dudv=\Delta^{ij}.
\end{align*}
Combining two equations above,
\begin{equation}\label{ex-tildeH}
	\E\tilde{H}^{\ep,ij}_t= \Gamma^{ij}+\Delta^{ij}=\bar{\Gamma}^{ij}.
\end{equation}
Using \eqref{H-approx-p} with $p=1,$
\[|\E H^{\ep,ij}_t-\bar{\Gamma}^{ij}|=|\E H^{\ep,ij}_t-\E\tilde{H}^{\ep,ij}_t|\leq\E|H^{\ep,ij}_t-\tilde{H}^{\ep,ij}_t|\leq C(e^{-t/\ep}+\sqrt{\ep}).\]
Therefore, for each fixed $t>0,$
\[\lim_{\ep\to 0}|\E H^{\ep,ij}_t-\bar{\Gamma}^{ij}|=0,\]
and
\[\int_0^T|\E H^{\ep,ij}_t-\bar{\Gamma}^{ij}|dt\leq\int_0^T C(e^{-t/\ep}+\sqrt{\ep})dt\leq C_T(\ep+\sqrt{\ep})\leq C_T\sqrt{\ep}.\]
\end{proof}

\noindent\emph{Proof of Theorem \ref{main2}:}
Since $(\bfxe)_{0<\ep\leq 1}$ is tight in $\mathscr{C}^{0,\al},$ there is a subsequence denoted again as $\bfxe$ and a probability measure $\nu$ on $\mathscr{C}^{0,\al}$ such that $\law(\bfxe)\to \nu$ weakly in $\mathscr{C}^{0,\al}.$ Suppose that $\nu$ is realized on some probability space $(\bar\Omega,\bar\F,\bar P).$ Next we identify $\bfx.$

By \eqref{SKxv}, \eqref{M-U-D-relation} and the fact $W^\ep=\sqrt{\ep}V^\ep,$
\begin{align*}
	dX^{\ep,j}_t&=V^{\ep,j}_tdt\\&=F^j(\xe_t)dt+dU^{\ep,j}_t-\ep dV^{\ep,j}_t\\&=F^j(\xe_t)dt+dM^{\ep,j}_t-\sqrt{\ep}d(D^{\ep,j}_t+W^{\ep,j}_t),
\end{align*}
so by \eqref{chain}, we have
\begin{equation*}
	\begin{aligned}[b]
		d\vp(\bfxe_t)&=\sum_ja_j(\bfxe_t)dX^{\ep,j}_t\\&=\sum_ja_j(\bfxe_t)[F^j(\xe_t)dt+dM^{\ep,j}_t-\sqrt{\ep}d(D^{\ep,j}_t+W^{\ep,j}_t)]\\&=\Phi(\bfxe_t,F(\xe_t))dt+\sum_ja_j(\bfxe_t)dM^{\ep,j}_t-\sqrt{\ep}\sum_ja_j(\bfxe_t)d(D^{\ep,j}_t+W^{\ep,j}_t),
	\end{aligned}
\end{equation*}
so
\begin{equation*}
	\begin{aligned}[b]
		&\quad\vp(\bfxe_t)-\vp(0)\\&=\intot\Phi(\bfxe_s,F(\xe_s))ds+\sum_j\intot a_j(\bfxe_s)dM^{\ep,j}_s-\sqrt{\ep}\sum_j\intot a_j(\bfxe_s)d(D^{\ep,j}_s+W^{\ep,j}_s).
	\end{aligned}
\end{equation*}
An integration-by-part and usage of \eqref{chain}
\begin{equation*}
	\begin{aligned}[b]
		&\quad\intot a_j(\bfxe_s)d(D^{\ep,j}_s+W^{\ep,j}_s)\\&=a_j(\bfxe_t)(D^{\ep,j}_t+W^{\ep,j}_t)-a_j(0)D^{\ep,j}_0-\intot(D^{\ep,j}_s+W^{\ep,j}_s)da_j(\bfxe_s)\\&=a_j(\bfxe_t)(D^{\ep,j}_t+W^{\ep,j}_t)-a_j(0)D^{\ep,j}_0-\intot(D^{\ep,j}_s+W^{\ep,j}_s)\sum_i\Theta^{ij}(\bfxe_s)dX^{\ep,i}_s\\&=a_j(\bfxe_t)(D^{\ep,j}_t+W^{\ep,j}_t)-a_j(0)D^{\ep,j}_0-\frac{1}{\sqrt{\ep}}\sum_i\intot(D^{\ep,j}_s+W^{\ep,j}_s)\Theta^{ij}(\bfxe_s)W^{\ep,i}_sds\\&=a_j(\bfxe_t)(D^{\ep,j}_t+W^{\ep,j}_t)-a_j(0)D^{\ep,j}_0-\frac{1}{\sqrt{\ep}}\sum_i\intot\Theta^{ij}(\bfxe_s)H^{\ep,ij}_sds.
	\end{aligned}
\end{equation*}
Combining two equations above and rearranging,
\begin{equation}\label{martingale-bfxe}
	\begin{aligned}[b]
		N^{\ep,X}_t:=&\varphi(\bfxe_t)-\varphi(0)
		+\sqrt\ep\sum_j\left[
		a_j(\bfxe_t)(D_t^{\ep,j}+W_t^{\ep,j})
		-a_j(0)D_0^j\right]\\
		&-\int_0^t\Phi(\bfxe_r,F(\xe_r))\,dr
		-\int_0^t\sum_{i,j}\Theta^{ij}(\bfxe_r)
		H_r^{\ep,ij}\,dr\\
		=&\sum_j\int_0^t a_j(\bfxe_r)\,dM_r^{\ep,j}.
	\end{aligned}
\end{equation}
Since $a_j$ is bounded and $\me$ is a martingale, the right-hand-side of \eqref{martingale-bfxe} is a martingale.
By Proposition \ref{prop-momentestimateforv} and \eqref{p-moment-D},
\begin{equation}\label{approx-X-martingale-1}
	 \sqrt{\ep}E|
	a_j(\bfx_t^\varepsilon)(D_t^{\ep,j}+W_t^{\ep,j})
	-a_j(0)D_0^j|\leq C\sqrt{\ep}\to 0,\quad 0\leq t\leq T,\,1\leq j\leq d.
\end{equation}
Recall $\tilde{H}^{\ep,ij}_t=\chi^i(t/\ep)(\chi^j(t/\ep)+D^j_{t/\ep}).$ Note that
\begin{equation}\label{newnoname1}
	\begin{aligned}[b]
		&\quad\Big|\intot\Theta^{ij}(\bfxe_r)(H^{\ep,ij}_r-\bar{\Gamma}^{ij})dr\Big|\\&\leq\Big|\intot\Theta^{ij}(\bfxe_r)(H^{\ep,ij}_r-\tilde{H}^{\ep,ij}_r)dr\Big|+\Big|\intot\Theta^{ij}(\bfxe_r)(\tilde{H}^{\ep,ij}_r-\bar{\Gamma}^{ij})dr\Big|\\&\leq \int_0^T|\Theta^{ij}(\bfxe_r)(H^{\ep,ij}_r-\tilde{H}^{\ep,ij}_r)|dr+\supt \Big|\intot\Theta^{ij}(\bfxe_r)(\tilde{H}^{\ep,ij}_r-\bar{\Gamma}^{ij})dr\Big|.
	\end{aligned}
\end{equation}
By \eqref{H-approx-p},
\begin{equation}\label{newnoname2}
	\begin{aligned}[b]
		&\E\int_0^T|\Theta^{ij}(\bfxe_r)(H^{\ep,ij}_r-\tilde{H}^{\ep,ij}_r)|dr\\&\leq \|\Theta\|_\infty\int_0^T \E|H^{\ep,ij}_r-\tilde{H}^{\ep,ij}_r|dr\\&\leq C\int_0^T (e^{-r/\ep}+\sqrt{\ep})\,dr\leq C\sqrt{\ep}.
	\end{aligned}
\end{equation}
Set $q^{ij}_t:=\chi^i(t)\chi^j(t)+\chi^i(t)D^j_t-\bar{\Gamma}^{ij},$ so that $\tilde{H}^{\ep,ij}_r-\bar{\Gamma}^{ij}=q^{ij}_{r/\ep}.$
Note that by \eqref{ex-tildeH},
\[\E q^{ij}_0=\E\chi^i(0)\chi^j(0)+\E\chi^i(0)D^j_0-\bar{\Gamma}^{ij}=\E\tilde{H}^{\ep,ij}_0-\bar{\Gamma}^{ij}=0,\]
and $q$ is a stationary and ergodic process. Proposition \ref{prop-momentestimate} implies that $(\bfxe)_{0<\ep\leq 1}$ is tight in $C([0,T];\mathbb{R}^{d+d^2}),$ and by continuous mapping theorem,  $(\Theta^{ij}(\bfxe))_{0<\ep\leq 1}$ is tight in $C([0,T]).$ It is straightforward to check that $(\chi,D)$ is jointly strictly stationary and ergodic, so  by Lemma \ref{lemma-vanishing-weight}, 
\begin{equation}\label{newnoname3}
	\E\supt \Big|\intot\Theta^{ij}(\bfxe_r)(\tilde{H}^{\ep,ij}_r-\bar{\Gamma}^{ij})dr\Big|\to 0.
\end{equation}
Combining \eqref{newnoname1}--\eqref{newnoname3},
\begin{equation}\label{approx-X-martingale-2}
	\E\supt\Big|\intot\Theta^{ij}(\bfxe_r)(H^{\ep,ij}_r-\bar{\Gamma}^{ij})dr\Big|\to 0,
\end{equation}
Set \[ e^{\ep,X}_t:=-\sqrt\ep\sum_j\left[
a_j(\bfxe_t)(D_t^{\ep,j}+W_t^{\ep,j})
-a_j(0)D_0^j\right]+\sum_{i,j}\intot\Theta^{ij}(\bfxe_r)(H^{\ep,ij}_r-\bar{\Gamma}^{ij})dr, \]
so that
\[M^{\ep,X}_t:=N^{\ep,X}_t+e^{\ep,X}_t=\vp(\bfxe_t)-\vp(0)-\intot\Phi(\bfxe_r,F(\xe_r))dr-\sum_{i,j}\intot\Theta^{ij}(\bfxe_r)\bar\Gamma^{ij}dr.\]
Define 
\begin{equation}\label{def-Lx}
	\LL_X\vp(x,X):=\Phi(x,X,F(x))+\sum_{i,j}\Theta^{ij}(x,X)\bar\Gamma^{ij},
\end{equation} so that the equation above is written as
\[M^{\ep,X}_t:=\vp(\bfxe_t)-\vp(0)-\intot\LL_X\vp(\bfxe_r)dr.\] By \eqref{approx-X-martingale-1} and \eqref{approx-X-martingale-2}, 
\[\lim\limits_{\ep\to 0}\E |e^{\ep,X}_t|=0,\quad t\geq0,\]
so by the same argument leading to \eqref{confirm-martingale-U}, we conclude that
\[t\mapsto\vp(\bfx_t)-\vp(0)-\intot\LL_X\vp(\bfx_r)dr \] is a $\F^\bfx$ martingale.
Set $\vp_1^i$ and $\vp_2^{ij}$ as in \eqref{phi1phi2}. A direct computation gives
\[\LL_X\vp_1^i(x,X)=F^i(x),\quad \LL_X\vp_2^{ij}(x,X)=x^iF^j(x)+\bar{\Gamma}^{ij},\] so a localization argument yields that
\begin{equation*}
	M^{X,i}_t:=X^i_t-\intot F^i(X_r)dr
\end{equation*}
and
\begin{equation*}
	N^{X,ij}_t:=\bx^{ij}_t-\intot X^i_rF^j(X_r)dr-\bar\Gamma^{ij}t
\end{equation*}
are two continuous local martingales. Note that
\[\bar\Gamma^{ij}+\bar\Gamma^{ji}=\Gamma^{ij}+\Gamma^{ji}+\Delta^{ij}+\Delta^{ji}=S^{ij}.\] By Remark \ref{remark-VL} and \eqref{qv},
\begin{align}\label{qv-X}
	[M^{X,i},M^{X,j}]_t&=S^{ij}t,\nonumber\\ [M^{X,k},N^{X,ij}]_t&=\intot S^{kj}X^i_rdr,\nonumber\\ [N^{X,ij},N^{X,kl}]_t&=\intot X^i_rX^k_rS^{jl}dr.
\end{align}
By L\'evy characterization theorem, $M^{X}_t=S^{1/2}B_t$ for some standard Brownian motion $B$  on a possibly enlarged probability space, so by definition of $M^{X},$
\[X_t=M^X_t+\intot F(X_r)dr=\intot F(X_r)dr+S^{1/2}B_t\]
Set $I^{X,ij}_t=\intot X^i_rdM^{X,j}_r.$ By \eqref{qv-X},
\begin{equation*}
	[I^{X,ij},I^{X,ij}]_t=\intot (X^i_r)^2d[M^{X,j}]_r=\intot (X^i_r)^2S^{jj}dr,
\end{equation*}
and 
\begin{equation*}
	[N^{X,ij},I^{X,ij}]_t=\intot X^i_rd[M^{X,j},N^{X,ij}]_r=\intot (X^i_r)^2S^{jj}dr.
\end{equation*}
Therefore, by \eqref{qv-X} again, we have
\begin{equation*}
	[N^{X,ij}-I^{X,ij}]_t=[N^{X,ij}]_t+[I^{X,ij}]_t-2[N^{X,ij},I^{X,ij}]_t=0.
\end{equation*}
Therefore, $N^{X,ij}=I^{X,ij}.$ Substituting into the definition of $N^{X,ij}_t,$ 
\begin{align*}
	\bx^{ij}_t&=N^{X,ij}_t+\intot X^i_rF^j(X_r)dr+\bar\Gamma^{ij}t\\&=\intot X^i_rdM^{X,j}_r+\intot X^i_rF^j(X_r)dr+\bar\Gamma^{ij}t.
\end{align*}
By definition of $M^X,$ $dM^{X,j}_r=dX^j_r-F^j(X_r)dr,$ so
\[\bx^{ij}_t=\intot X^i_rdX^j_r+\bar\Gamma^{ij}t.\]
A direct computation using Chen relation,
\[\bx^{ij}_{s,t}=\int_s^t X^i_{s,r}\circ dX^j_r+\frac{\bar\Gamma^{ij}-\bar\Gamma^{ji}}{2}(t-s).\]
Since $F$ is Lipschitz, $X$ is unique in law. Therefore, every converging subsequence of $\bfxe$ has the same limit in the sense of distribution. The desired weak convergence follows from tightness of $(\bfxe)_{0<\ep\leq 1}.$
 
\qed

\vspace{\baselineskip}
\noindent\emph{Proof of Theorem \ref{main3}:} 

Fix $z\in\mathbb{R}^e.$ Set \[\mathcal I: (\mathscr{C}^{0,\al},\rho_\al)\mapsto(C^\al,\|\cdot\|_\al),\quad \bfy\mapsto w, \] with $w$ the solution of RDE
\[dw_t=b(w_t)dt+\sigma(w_t)d\bfy_t,\quad w_0=z.\]
By continuity of It\^o--Lyons map, $\mathcal I$ is continuous, so by Theorem \ref{main2} and continuous mapping theorem,
\[\ze=\mathcal I(\bfxe)\to \mathcal I(\bfx)=Z\quad\text{in}\,\,\mathcal C^\al\,\,\text{weakly,}\]
where $Z$ is the solution of
\begin{equation}\label{RDE}
	dZ_t=b(Z_t)dt+\sigma(Z_t)d\bfx_t,\quad Z_0=z.
\end{equation}
Now we transfer the RDE \eqref{RDE} to SDE. By the definition of RDE, 
\begin{equation*}
	Z_t-Z_s=b(Z_s)(t-s)+\sigma(Z_s)X_{s,t}+D\sigma(Z_s)\sigma(Z_s)\mathbb{X}_{s,t}+R_{s,t}, \quad 0\leq s\leq t\leq T,
\end{equation*}
where the remainder $R$ satisfying $\lim\limits_{\delta\to 0}\sup\limits_{\mathcal{P}\subset [0,T],|\mathcal P|<\delta}\sum\limits_{[s,t]\in\mathcal{P}}|R_{s,t}|=0.$
Writing this equation in its component,
\begin{equation}\label{componentRDE}
Z^i_t-Z^i_s=b^i(Z_s)(t-s)+\sum_{j}\sigma_{ij}(Z_s)X^j_{s,t}+\sum_{k,l,j}\partial_k\sigma_{ij}(Z_s)\sigma_{kl}(Z_s)\mathbb{X}^{lj}_{s,t}+R^i_{s,t}.
\end{equation}
By \eqref{limitofX},
\begin{eqnarray}\label{noname16}
	\mathbb{X}^{lj}_{s,t}&=&\int_s^t X^l_{s,r}\circ dX^j_r+\frac{\bar{\Gamma}_{lj}-\bar{\Gamma}_{jl}}{2}(t-s)\nonumber\\&=&\int_s^t X^l_{s,r}dX^j_r+\frac{1}{2}[X^l_{s,\cdot},X^j]_{s,t}+\frac{\bar{\Gamma}_{lj}-\bar{\Gamma}_{jl}}{2}(t-s)\nonumber\\&=&\int_s^t X^l_{s,r}dX^j_r+\frac{1}{2}[\sum_k S_{lk}^{1/2}B^k_{s,\cdot},\sum_{k=1}^d S_{jk}^{1/2}B^k_\cdot]_{s,t}+\frac{\bar{\Gamma}_{lj}-\bar{\Gamma}_{jl}}{2}(t-s)\nonumber\\&=&\int_s^t X^l_{s,r}dX^j_r+\frac{1}{2}\sum_k S^{1/2}_{lk}S^{1/2}_{jk}(t-s)+\frac{\bar{\Gamma}_{lj}-\bar{\Gamma}_{jl}}{2}(t-s).
\end{eqnarray}
Define matrix $G$ by
\begin{equation}\label{defofG}
	G^{lj}:=\frac{1}{2}(\sum_k S^{1/2}_{lk}S^{1/2}_{jk}+\bar{\Gamma}_{lj}-\bar{\Gamma}_{jl}),
\end{equation}
then \eqref{noname16} is written as
\begin{equation}\label{noname17}
	\mathbb{X}^{lj}_{s,t}=\int_s^t X^l_{s,r}dX^j_r+G^{lj}(t-s)=\mathbb{X}^{\text{It\^o},lj}_{s,t}+G^{lj}(t-s).
\end{equation}
Combining \eqref{componentRDE} and \eqref{noname17},
\begin{eqnarray}\label{noname19}
	&&Z^i_t-Z^i_s\nonumber\\&=&b^i(Z_s)(t-s)+\sum_{j}\sigma_{ij}(Z_s)X^j_{s,t}+\sum_{k,l,j}\partial_k\sigma_{ij}(Z_s)\sigma_{kl}(Z_s)\mathbb{X}^{\text{It\^o},lj}_{s,t}\nonumber\\&&+\sum_{k,l,j}\partial_k\sigma_{ij}(Z_s)\sigma_{kl}(Z_s)G^{lj}(t-s)+R^i_{s,t}\nonumber\\&=&[b^i(Z_s)+\sum_{k,l,j}\partial_k\sigma_{ij}(Z_s)\sigma_{kl}(Z_s)G^{lj}](t-s)+\sum_{j}\sigma_{ij}(Z_s)X^j_{s,t}\nonumber\\&&+\sum_{k,l,j}\partial_k\sigma_{ij}(Z_s)\sigma_{kl}(Z_s)\mathbb{X}^{\text{It\^o},lj}_{s,t}+R^i_{s,t}.
\end{eqnarray}
By equivalence theorem of SDE and RDE \cite[Proposition 6.9]{FZ18}, the solution of \eqref{noname19} is indistinguishable to the SDE
\begin{equation}\label{noname19+}
	dZ^i_t=[b^i(Z_t)+\sum_{k,l,j}\partial_k\sigma_{ij}(Z_t)\sigma_{kl}(Z_t)G^{lj}]dt+\sum_j\sigma_{ij}(Z_t)dX^j_t.
\end{equation} 
But $$dX^j_t=F^j(X_t)dt+\sum_k S^{1/2}_{jk}dB^k_t,$$ substituting which into \eqref{noname19+} yields
\begin{eqnarray*}
	dZ^i_t=[b^i(Z_t)+\sum_{k,l,j}\partial_k\sigma_{ij}(Z_t)\sigma_{kl}(Z_t)G^{lj}+\sum_j\sigma_{ij}(Z_t)F^j(X_t)]dt+\sum_{j,k} S^{1/2}_{jk}\sigma_{ij}(Z_t)dB^k_t.
\end{eqnarray*}\qed

\appendix
\section{Appendix}\label{appen}
In the appendix we give lemmas which are useful in identifying the limit path $\mathbf U$ and $\bfx.$
\begin{lemma}\label{lemma-vanishing-weight}
	Let $q$ be an adapted c\`adl\`ag stationary ergodic process with $\E q_0=0$ and $q(0)\in L^p$ for some $p>1.$ Let $(a^\ep)_{0<\ep\leq 1}$ be a family of $C([0,T])$-valued random variables. Assume  $(a^\ep)_{0<\ep\leq 1}$ is tight in $C([0,T]),$ and
	\[|a^\ep(t,\omega)|\leq C,\quad \ep\in (0,1],\,t\in[0,T],\,\omega\in\Omega.\]
	We have
	\begin{equation*}
		\E\supt\Big|\intot a^\ep_r q(r/\ep)dr\Big|\to 0.
	\end{equation*}
\end{lemma}
\begin{proof}
	Set $A(u):=\int_0^u q(v)dv,$ $Q^\ep:=\supt|\intot q(r/\ep)dr|.$ A change-of-variable gives
	\[Q^\ep:=\supt\ep \Big|\int_0^{t/\ep}q(r)dr\Big|=\sup_{0\leq u\leq T/\ep}\ep\Big|\int_0^u q(r)dr\Big|=\sup_{0\leq u\leq T/\ep}\ep|A(u)|.\]
	By ergodic theorem, as $u\to+\infty,$ $A(u)/u\to \E q(0)=0$ a.s., so for each $\gamma>0,$ there is a random $R>0$ such that $|A(u)|\leq \gamma u$ for all $u\geq R.$ Setting $K_R:=\sup\limits_{0\leq u\leq R}|A(u)|<\infty$ a.s., we have
	\begin{align*}
		Q^\ep&=\sup_{0\leq u\leq T/\ep}\ep|A(u)|\\&\leq \sup_{0\leq u\leq R}\ep|A(u)|+\sup_{R<u\leq T/\ep}\ep|A(u)|\\&\leq \ep K_R+\ep\gamma \sup_{R<u\leq T/\ep}u=\ep K_R+\gamma T.
	\end{align*}
	Letting $\ep\to0,$ we have
	\[\limsup_{\ep\to0}Q^\ep\leq \gamma T, \quad\text{a.s.}\]
	Since $\gamma$ is arbitrary, there must be
	\begin{equation}\label{vanishing-Qe}
		\lim\limits_{\ep\to 0}Q^\ep=0,\quad\text{a.s.}
	\end{equation}
	Set $p^\prime$ satisfying that $1/p+1/p^\prime=1.$ By H\"older inequality, for each $h>0,$
	\begin{align*}
		\E Q^\ep&=\E Q^\ep\mathbf{1}_{\{Q^\ep\leq h\}}+\E Q^\ep\mathbf{1}_{\{Q^\ep> h\}}\\&\leq h+\|Q^\ep\|_p\|\mathbf{1}_{\{Q^\ep>h\}}\|_{p^\prime}\\&\leq h+\|Q^\ep\|_pP(Q^\ep>h)^{1/p^\prime}
	\end{align*}
	Note that by stationarity,
	\[\|Q^\ep\|_p\leq\int_0^T\|q(r/\ep)\|_pdr\leq T\|q(0)\|_p,\]
	and as a consequence of \eqref{vanishing-Qe},
	\[\lim\limits_{\ep\to 0} P(Q^\ep>h)=0,\]
	so we can continue the estimate above to obtain
	\[\limsup\limits_{\ep\to 0}\E Q^\ep\leq h.\]
	Since $h>0$ is arbitrary, we must have
	\begin{equation}\label{vanshing-ex-Qe}
		\lim\limits_{\ep\to0}\E Q^\ep=0.
	\end{equation}
	Now set $\omega(f,\delta):=\sup\limits_{0\leq s,t\leq T,|t-s|<\delta}|f(t)-f(s)|$ for each $f\in C([0,T]).$ For each $\kappa>0,$ since $(a^\ep)$ is tight in $C([0,T])$, there is compact set $K\subset C([0,T])$ such that
	\begin{equation}\label{tight-small-prob}
		P(a^\ep\notin K)<\kappa, \quad 0<\ep\leq 1.
	\end{equation}
	By Arzel\`a-Ascoli theorem, $\{f:f\in K\}$ is equi-continuous, so
	\begin{equation}\label{vanishing-omega}
		\sup_{f\in K}\omega(f,\delta)\to 0,\quad\delta\to 0.
	\end{equation}
	Besides, since $(a^\ep)$ is bounded in $C([0,T]),$ we also have
	\begin{equation}\label{bound-omega}
		\omega(a^\ep,\delta)\leq C,\quad \ep\in (0,1],\,\delta\in (0,T].
	\end{equation}
	By \eqref{tight-small-prob} and \eqref{bound-omega},
	\begin{align*}
		&\quad \E\omega(a^\ep,\delta)^{p^\prime}\\&=\E[\omega(a^\ep,\delta)^{p^\prime}\mathbf{1}_{\{a^\ep\in K\}}]+\E[\omega(a^\ep,\delta)^{p^\prime}\mathbf{1}_{\{a^\ep\notin K\}}]\\&\leq \sup_{f\in K}\omega(f,\delta)^{p^\prime}+C_{p^\prime} P(a^\ep\notin K)\leq \sup_{f\in K}\omega(f,\delta)^{p^\prime}+C_{p^\prime}\kappa.
	\end{align*}
	The right-hand-side of the above equation is independent of $\ep$, so by \eqref{vanishing-omega},
	\[\limsup_{\delta\to 0}\supe\E\omega(a^\ep,\delta)^{p^\prime}\leq C_{p^\prime}\kappa.\]
	Since $\kappa>0$ is arbitrary, there must be
	\begin{equation}\label{vanish-omega}
		\lim\limits_{\delta\to 0}\supe\E\omega(a^\ep,\delta)^{p^\prime}=0.
	\end{equation}
	Now partition $[0,T]$ into $n$ intervals with the same length:
	\[ \mathcal{P}_n=\{0=t_0<...<t_n=T\}. \]
	Define \[a^{\ep,n}_r:=a^\ep_{t_k},\quad t_k\leq r<t_{k+1},\,k=0,1,\cdots, n-1.\]	
	For each $t\in [0,T],$ assume $t\in [t_m,t_{m+1}],$ we have
	\begin{align*}
		&\quad\intot a^{\ep,n}_rq(r/\ep)dr\\&=\sum_{k=0}^{m-1}\int_{t_k}^{t_{k+1}}a^{\ep,n}_r q(r/\ep)dr+\int_{t_m}^t a^{\ep,n}_rq(r/\ep)dr\\&=\sum_{k=0}^{m-1}a^\ep_{t_k}\int_{t_k}^{t_{k+1}}q(r/\ep)dr+a^\ep_{t_m}\int_{t_m}^t q(r/\ep)dr.
	\end{align*}
	Note for all $0\leq u\leq v\leq T,$
	\[\Big|\int_u^v q(r/\ep)dr\Big|\leq \Big|\int_0^u q(r/\ep)dr\Big|+\Big|\int_0^v q(r/\ep)dr\Big|\leq 2Q^\ep,\] so by boundedness of $a^\ep,$
	\[\Big|\intot a^{\ep,n}_rq(r/\ep)dr\Big|\leq2\sum_{k=0}^{m-1}|a^\ep_{t_k}|Q^\ep+2|a^\ep_{t_m}|Q^\ep\leq 4CnQ^\ep,\]
	the right-hand-side of which is independent of $t.$ Taking supremum and expectation,
	\begin{equation}\label{part-1}
		\E\supt\Big|\intot a^{\ep,n}_rq(r/\ep)dr\Big|\leq 4Cn\E Q^\ep 
	\end{equation}
	Set $K^\ep:=\int_0^T |q(r/\ep)|dr,$ so that $\|K^\ep\|_p\leq \int_0^T \|q(r/\ep)\|_pdr\leq T\|q(0)\|_p=C_{p,T}.$ Note that
	\begin{align*}
		&\quad\Big|\intot (a^\ep_r-a^{\ep,n}_r)q(r/\ep)dr\Big|\\&\leq\int_0^T|a^\ep_r-a^{\ep,n}_r|\cdot|q(r/\ep)|dr\\&\leq\int_0^T\omega(a^\ep,T/n)|q(r/\ep)|dr=\omega(a^\ep,T/n)K^\ep,
	\end{align*}
	the right-hand-side of which is independent of $t.$ Taking supremum and expectation and applying H\"older inequality,
	\begin{equation}\label{part-2}
		\begin{aligned}[b]
			&\quad\E\supt\Big|\intot (a^\ep_r-a^{\ep,n}_r)q(r/\ep)dr\Big|\\&\leq \E[\omega(a^\ep,T/n)K^\ep]\\&\leq\|K^\ep\|_p\|\omega(a^\ep,T/n)\|_{p^\prime}\\&\leq C_{p,T}\|\omega(a^\ep,T/n)\|_{p^\prime}.
		\end{aligned}
	\end{equation}
	Combining \eqref{part-1} and \eqref{part-2},
	\begin{equation*}
		\E\supt\Big|\intot a^\ep_rq(r/\ep)dr\Big|\leq 4Cn\E Q^\ep+ C_{p,T}\|\omega(a^\ep,T/n)\|_{p^\prime}	.
	\end{equation*}
	By \eqref{vanshing-ex-Qe},
	\[\limsup_{\ep\to0}	\E\supt\Big|\intot a^\ep_rq(r/\ep)dr\Big|\leq C_{p,T}\supe\|\omega(a^\ep,T/n)\|_{p^\prime}.\]
	Letting $n\to\infty$ and applying \eqref{vanish-omega},
	\[\lim\limits_{\ep\to 0}\E\supt\Big|\intot a^\ep_rq(r/\ep)dr\Big|= 0.\]

\end{proof}
\begin{lemma}
	For each continuously differentiable function $f(x,X):\mathbb{R}^d\times\mathbb{R}^{d\times d}\to\mathbb{R},$ set
	\begin{equation*}
		V_a f(x,X):=\partial_{x_a}f(x,X)+\sum_{p=1}^{d}x_p\partial_{X_{pa}}f(x,X),\quad 1\leq a\leq d,
	\end{equation*}
	and assume that the operator $\LL$ is of the form
	\begin{equation}\label{def-L}
		\LL:=\sum_{a=1}^{d}h^a(x)V_a+\sum_{a,b=1}^{d}C^{ab}V_aV_b,
	\end{equation}
	where $h^a:\mathbb{R}^d\to\mathbb{R}$ is another locally bounded Borel function, and $C=(C^{ab})_{1\leq a,b\leq d}$ is a constant matrix. Then for $f,g\in C^2(\mathbb{R}^d\times\mathbb{R}^{d\times d}).$
	\begin{equation}\label{L-fg}
		\LL(fg)-f\LL g-g\LL f=\sum_{a,b=1}^d (C^{ab}+C^{ba})V_afV_bg.
	\end{equation}
\end{lemma}
\begin{proof}
	Notice that
	\begin{equation}\label{V-fg}
		V_a(fg)=fV_ag+gV_af,
	\end{equation}
	so
	\begin{equation}\label{V-V-fg}
		V_aV_b(fg)=fV_aV_bg+V_bgV_af+gV_aV_bf+V_bfV_ag.
	\end{equation}
	Equation \eqref{L-fg} is obtained by combining \eqref{def-L}, \eqref{V-fg} and \eqref{V-V-fg}.
\end{proof}
\begin{remark}\label{remark-VL}
	(i) A direct computation gives
	\[V_ax_k=\delta_{ak},\quad V_aX_{ij}=x^i\delta_{aj},\] and
	\[V_aV_bx_k=0,\quad V_aV_bX_{ij}=\delta_{ai}\delta_{bj}.\]
	
	(ii)  It is straightforward to check that $\LL_U$ defined in \eqref{def-Lu} is of the form \eqref{def-L} with $h^a\equiv 0$ and $C^{ab}=A^{ab},$ and $\LL_X$ defined in \eqref{def-Lx} is of the form \eqref{def-L} with $h^a=F^a$ and $C^{ab}=\bar\Gamma^{ab}.$
\end{remark}
\begin{lemma}
Let $Y$ be a $\mathbb{R}^d\times\mathbb{R}^{d\times d}$-valued continuous process on some filtered probability space such that for each $f\in C_c^\infty(\mathbb{R}^d\times\mathbb{R}^{d\times d}),$ \[M^f_t:=f(Y_t)-f(Y_0)-\intot\LL f(Y_r)dr \] is a martingale. Then we have
\begin{equation}\label{qv}
	[M^f,M^g]_t=\sum_{a,b=1}^d\intot (C^{ab}+C^{ba})V_af(Y_r)V_bg(Y_r) dr.
\end{equation}
\end{lemma}
\begin{proof}
Note that 
\begin{equation}\label{diff-martingale}
	df(Y_t)=\LL f(Y_t)dt+dM^f_t.
\end{equation}
For another $g\in C_c^\infty(\mathbb{R}^d\times\mathbb{R}^{d\times d}),$ by \eqref{diff-martingale},
\begin{equation}\label{diff-product}
	\begin{aligned}[b]
		&\quad d[f(Y_t)g(Y_t)]\\&=f(Y_t)dg(Y_t)+g(Y_t)df(Y_t)+d[f(Y),g(Y)]_t\\&=f(Y_t)dg(Y_t)+g(Y_t)df(Y_t)+d[M^f,M^g]_t\\&=f(Y_t)(\LL g(Y_t)dt+dM^g_t)+g(Y_t)(\LL f(Y_t)dt+dM^f_t)+d[M^f,M^g]_t\\&=(f\LL g+g\LL f)(Y_t)dt+(f(Y_t)dM^g_t+g(Y_t)dM^f_t)+d[M^f,M^g]_t.
	\end{aligned}
\end{equation}
Replacing $f$ in \eqref{diff-martingale} by $fg,$
\[d[f(Y_t)g(Y_t)]=[\LL(fg)](Y_t)dt+dM^{fg}_t. \]
Combined with \eqref{diff-product},
\[dM^{fg}_t-f(Y_t)dM^g_t-g(Y_t)dM^f_t=d[M^f,M^g]_t+(f\LL g+g\LL f-\LL(fg))(Y_t)dt.\]
Integrating on both sides,
\[M^{fg}_t-\intot f(Y_s)dM^g_s-\intot g(Y_s)dM^f_s=[M^f,M^g]_t+\intot(f\LL g+g\LL f-\LL(fg))(Y_r)dr, \]
the left-hand-side being a continuous local martingale and the right-hand-side being a finite variation process. Obviously they both started at $0,$ so they must be both $0.$ Therefore,
\begin{equation*}
	[M^f,M^g]_t=\intot(\LL(fg)-f\LL g-g\LL f)(Y_r)dr.
\end{equation*} 
The proof is finished by applying \eqref{L-fg}.

\end{proof}
\section*{Statements and Declarations}
\subsection*{Competing Interests}
The authors declare that have no competing interests as defined by Springer, or other interests that might be perceived to influence the results and/or discussion reported in this paper.

	%\newpage
	%\bibliographystyle{abbrv}
	%\addcontentsline{toc}{section}{References}
	%\bibliography{bibli4}	

\end{document}